\documentclass[11pt]{article}
\usepackage{amsthm}
\usepackage{amsmath}
\usepackage{amssymb}
\usepackage{makeidx}
\theoremstyle{definition}
\newtheorem{definition}{Definition}[section]
\newtheorem{notation}[definition]{Notation}
\newtheorem{example}[definition]{Example}

\theoremstyle{plain}
\newtheorem{theorem}[definition]{Theorem}
\newtheorem{lemma}[definition]{Lemma}
\newtheorem{proposition}[definition]{Proposition}
\newtheorem{corollary}[definition]{Corollary}
\newtheorem{remark}[definition]{Remark}
\newtheorem{conjecture}[definition]{Conjecture}

\newcommand{\beq}{\begin{equation}}

\newcommand{\eeq}{\end{equation}}

\newcommand{\bdfn}{\begin{definition}}

\newcommand{\edfn}{\end{definition}}

\newcommand{\bthm}{\begin{theorem}}

\newcommand{\ethm}{\end{theorem}}

\newcommand{\bprop}{\begin{proposition}}

\newcommand{\eprop}{\end{proposition}}

\newcommand{\bcor}{\begin{corollary}}

\newcommand{\ecor}{\end{corollary}}

\newcommand{\blem}{\begin{lemma}}

\newcommand{\elem}{\end{lemma}}

\newcommand{\bex}{\begin{example}}

\newcommand{\eex}{\end{example}}

\newcommand{\bxc}{\begin{exercise}}

\newcommand{\exc}{\end{exercise}}

\newcommand{\bntn}{\begin{notation}}

\newcommand{\entn}{\end{notation}}

\newcommand{\be}{\begin{enumerate}}

\newcommand{\ee}{\end{enumerate}}

\newcommand{\bce}{\begin{center}}

\newcommand{\ece}{\end{center}}

\newcommand{\bi}{\begin{itemize}}

\newcommand{\ei}{\end{itemize}}

\newcommand{\bt}{\begin{tabular}}

\newcommand{\et}{\end{tabular}}

\newcommand{\ra}{\rightarrow}

\newcommand{\si}{\wedge}

\newcommand{\sau}{\vee}

\newcommand{\ba}{\begin{array}} 

\newcommand{\ea}{\end{array}}

\numberwithin{equation}{section}

\def\N{{\mathbb N}}

\newcommand {\bua} {\begin{eqnarray*}}

\newcommand {\eua} {\end {eqnarray*}}

\begin{document}
\title{Co-Stone Residuated Lattices}
\author{Claudia MURE\c{S}AN\thanks{Dedicated to the memory of my dear grandmother, Elena Mircea}\\ \footnotesize University of Bucharest\\ \footnotesize Faculty of Mathematics and Computer Science\\ \footnotesize Academiei 14, RO 010014, Bucharest, Romania\\ \footnotesize Emails: c.muresan@yahoo.com, cmuresan11@gmail.com}
\date{}
\maketitle

\begin{abstract}
In this paper we present some applications of the reticulation of a residuated lattice, in the form of a transfer of properties between the category of bounded distributive lattices and that of residuated lattices through the reticulation functor. The results we are presenting are related to co-Stone algebras; among other applications, we transfer a known characterization of $m$-co-Stone bounded distributive lattices to residuated lattices and we prove that the reticulation functor for residuated lattices preserves the strongly co-Stone hull.
\end{abstract}

\section{Introduction}

\hspace*{11pt} In \cite{eu1} we gave an axiomatic purely algebraic definition for the reticulation of a residuated lattice, which turned out to be very useful in practice. In this work we present several applications of the reticulation, related to co-Stone algebras, applications in the form of transfers of properties between the category of bounded distributive lattices and the category of residuated lattices through the reticulation functor.

The co-Stone structures were introduced by us as being dual notions to Stone structures. In Section \ref{preliminaries} we introduce their definition, along with recalling some definitions and results that the reader might find necessary for understanding the results in the next sections.

In Section \ref{co-Stonealgebras} we prove the fact that a residuated lattice is co-Stone iff its reticulation is co-Stone and the same is valid for strongly co-Stone structures, then we obtain a structure theorem for $m$-co-Stone residuated lattices, by transferring through the reticulation a known characterization of $m$-co-Stone bounded distributive lattices to residuated lattices. This is the first major example of a result that can be transferred through the reticulation functor from the category of bounded distributive lattices to the category of residuated lattices. It also permits us to state that a residuated lattice is $m$-co-Stone iff its reticulation is $m$-co-Stone. We then bring an argument for our choice of the definition of the co-Stone structures over another definition for them that can be found in mathematical litterature, for instance in \cite{rcig}: the fact that the notion with our definition is transferrable through the reticulation (while the alternate one is not and does not coincide with ours). 

In Section \ref{hull}, we construct the strongly co-Stone hull of a residuated lattice, conjecture a universality property for it, show that it is preserved by the reticulation functor and exemplify its calculation for a finite residuated lattice.

In future papers we will continue our research on the transfer of pro\-per\-ties between the category of bounded distributive lattices and that of residuated lattices through the reticulation functor. This transfer of pro\-per\-ties between different categories is the very purpose of the reticulation.

\section{Preliminaries}
\label{preliminaries}

\begin{definition}
A {\em residuated lattice} is an algebraic structure $(A,\vee ,\wedge ,\odot ,$\linebreak $\rightarrow ,0,1)$, with the first 4 operations binary and the last two constant, such that $(A,\vee ,\wedge ,0,1)$ is a bounded lattice, $(A,\odot ,1)$ is a commutative monoid and the following property, called {\em the law of residuation}, is satisfied: for all $a,b,c\in A$, $a\leq b\rightarrow c$ iff $a\odot b\leq c$, where $\leq $ is the partial order of the lattice $(A,\vee ,\wedge ,0,1)$.
\end{definition}

Here are some examples of residuated lattices that we will use in the sequel, for illustrating various properties and various classes of residuated lattices.

\begin{example}
\label{lrex0}
{\rm  \cite[Section 11.1]{ior1}, \cite{ior}} $A=\{0,a,b,c,1\}$, with the bounded lattice structure given by the Hasse diagram below and the operations that succeed it, is a residuated lattice.

\begin{center}
\begin{picture}(70,95)(0,0)
\put(30,10){\line(-1,1){20}}
\put(30,10){\line(1,1){20}}
\put(30,50){\line(-1,-1){20}}
\put(30,50){\line(1,-1){20}}
\put(30,50){\line(0,1){20}}
\put(30,10){\circle*{3}}
\put(10,30){\circle*{3}}
\put(50,30){\circle*{3}}
\put(30,50){\circle*{3}}
\put(30,70){\circle*{3}}
\put(28,0){$0$}
\put(1,27){$a$}
\put(54,27){$b$}
\put(34,48){$c$}
\put(28,75){$1$}
\end{picture}
\end{center}

\begin{center}
\begin{tabular}{c|ccccc}
$\rightarrow $ & $0$ & $a$ & $b$ & $c$ & $1$ \\ \hline
$0$ & $1$ & $1$ & $1$ & $1$ & $1$ \\
$a$ & $b$ & $1$ & $b$ & $1$ & $1$ \\
$b$ & $a$ & $a$ & $1$ & $1$ & $1$ \\
$c$ & $0$ & $a$ & $b$ & $1$ & $1$ \\
$1$ & $0$ & $a$ & $b$ & $c$ & $1$
\end{tabular}
\end{center}

\noindent and $\odot =\wedge $.
\end{example}

\begin{example}
\label{lrex0,5}
$A=\{0,a,b,c,1\}$, with the lattice structure and the operations presented below, is a residuated lattice:

\begin{center}
\begin{picture}(80,90)(0,0)
\put(38,3){$0$}
\put(40,15){\circle*{3}}
\put(40,15){\line(0,1){20}}
\put(40,35){\circle*{3}}
\put(45,30){$a$}

\put(40,35){\line(-1,1){20}}
\put(40,35){\line(1,1){20}}
\put(20,55){\circle*{3}}

\put(10,52){$b$}
\put(60,55){\circle*{3}}
\put(65,52){$c$}
\put(20,55){\line(1,1){20}}
\put(60,55){\line(-1,1){20}}
\put(40,75){\circle*{3}}
\put(38,80){$1$}
\end{picture}
\end{center}

\begin{center}
\begin{tabular}{c|ccccc}
$\rightarrow $ & $0$ & $a$ & $b$ & $c$ & $1$ \\ \hline
$0$ & $1$ & $1$ & $1$ & $1$ & $1$ \\
$a$ & $0$ & $1$ & $1$ & $1$ & $1$ \\

$b$ & $0$ & $c$ & $1$ & $c$ & $1$ \\
$c$ & $0$ & $b$ & $b$ & $1$ & $1$ \\

$1$ & $0$ & $a$ & $b$ & $c$ & $1$
\end{tabular}
\end{center}

\noindent and $\odot =\wedge $.
\end{example}

\begin{example}
\label{lrex3}
{\rm \cite{ior1}} $A=\{0,a,b,c,d,1\}$, with the structure described below, is a residuated lattice.

\begin{center}
\begin{picture}(100,90)(0,0)
\put(40,11){\circle*{3}}
\put(38,0){$0$}
\put(20,31){\circle*{3}}
\put(10,28){$a$}
\put(60,31){\circle*{3}}
\put(65,28){$b$}
\put(40,51){\circle*{3}}
\put(30,51){$c$}
\put(80,51){\circle*{3}}

\put(85,48){$d$}
\put(60,71){\circle*{3}}

\put(58,76){$1$}
\put(40,11){\line(1,1){40}}
\put(20,31){\line(1,1){40}}
\put(40,11){\line(-1,1){20}}
\put(60,31){\line(-1,1){20}}
\put(80,51){\line(-1,1){20}}
\end{picture}
\end{center}

\begin{center}
\begin{tabular}{cc}
\begin{tabular}{c|cccccc}

$\rightarrow $ & $0$ & $a$ & $b$ & $c$ & $d$ & $1$ \\ \hline
$0$ & $1$ & $1$ & $1$ & $1$ & $1$ & $1$ \\
$a$ & $d$ & $1$ & $d$ & $1$ & $d$ & $1$ \\
$b$ & $c$ & $c$ & $1$ & $1$ & $1$ & $1$ \\
$c$ & $b$ & $c$ & $d$ & $1$ & $d$ & $1$ \\

$d$ & $a$ & $a$ & $c$ & $c$ & $1$ & $1$ \\
$1$ & $0$ & $a$ & $b$ & $c$ & $d$ & $1$
\end{tabular}
& \hspace*{11pt}
\begin{tabular}{c|cccccc}
$\odot $ & $0$ & $a$ & $b$ & $c$ & $d$ & $1$ \\ \hline
$0$ & $0$ & $0$ & $0$ & $0$ & $0$ & $0$ \\
$a$ & $0$ & $a$ & $0$ & $a$ & $0$ & $a$ \\
$b$ & $0$ & $0$ & $0$ & $0$ & $b$ & $b$ \\
$c$ & $0$ & $a$ & $0$ & $a$ & $b$ & $c$ \\
$d$ & $0$ & $0$ & $b$ & $b$ & $d$ & $d$ \\
$1$ & $0$ & $a$ & $b$ & $c$ & $d$ & $1$
\end{tabular}
\end{tabular}
\end{center}

\end{example}

\begin{example}
\label{lrex4}
{\rm \cite{ior1}} $A=\{0,a,b,c,d,e,f,g,1\}$, with the following structure, is a residuated lattice:

\begin{center}
\begin{picture}(120,110)(0,0)
\put(60,15){\circle*{3}}
\put(58,3){$0$}
\put(60,15){\line(-1,1){40}}
\put(40,35){\circle*{3}}
\put(31,30){$a$}
\put(60,15){\line(1,1){40}}
\put(80,35){\circle*{3}}
\put(84,30){$c$}
\put(40,35){\line(1,1){40}}
\put(80,35){\line(-1,1){40}}
\put(60,55){\circle*{3}}
\put(57,42){$d$}
\put(20,55){\circle*{3}}

\put(11,52){$b$}
\put(40,75){\circle*{3}}
\put(31,75){$e$}
\put(100,55){\circle*{3}}
\put(103,52){$f$}
\put(80,75){\circle*{3}}
\put(84,75){$g$}
\put(60,95){\line(-1,-1){40}}
\put(60,95){\line(1,-1){40}}
\put(60,95){\circle*{3}}
\put(58,100){$1$}
\end{picture}
\end{center}

\begin{center}
\begin{tabular}{c|ccccccccc}
$\rightarrow $ & $0$ & $a$ & $b$ & $c$ & $d$ & $e$ & $f$ & $g$ & $1$ \\ \hline
$0$ & $1$ & $1$ & $1$ & $1$ & $1$ & $1$ & $1$ & $1$ & $1$ \\

$a$ & $g$ & $1$ & $1$ & $g$ & $1$ & $1$ & $g$ & $1$ & $1$ \\
$b$ & $f$ & $g$ & $1$ & $f$ & $g$ & $1$ & $f$ & $g$ & $1$ \\

$c$ & $e$ & $e$ & $e$ & $1$ & $1$ & $1$ & $1$ & $1$ & $1$ \\
$d$ & $d$ & $e$ & $e$ & $g$ & $1$ & $1$ & $g$ & $1$ & $1$ \\
$e$ & $c$ & $d$ & $e$ & $f$ & $g$ & $1$ & $f$ & $g$ & $1$ \\

$f$ & $b$ & $b$ & $b$ & $e$ & $e$ & $e$ & $1$ & $1$ & $1$ \\

$g$ & $a$ & $b$ & $b$ & $d$ & $e$ & $e$ & $g$ & $1$ & $1$ \\
$1$ & $0$ & $a$ & $b$ & $c$ & $d$ & $e$ & $f$ & $g$ & $1$

\end{tabular}
\end{center}

\begin{center}
\begin{tabular}{c|ccccccccc}
$\odot $ & $0$ & $a$ & $b$ & $c$ & $d$ & $e$ & $f$ & $g$ & $1$ \\ \hline
$0$ & $0$ & $0$ & $0$ & $0$ & $0$ & $0$ & $0$ & $0$ & $0$ \\

$a$ & $0$ & $0$ & $a$ & $0$ & $0$ & $a$ & $0$ & $0$ & $a$ \\
$b$ & $0$ & $a$ & $b$ & $0$ & $a$ & $b$ & $0$ & $a$ & $b$ \\
$c$ & $0$ & $0$ & $0$ & $0$ & $0$ & $0$ & $c$ & $c$ & $c$ \\
$d$ & $0$ & $0$ & $a$ & $0$ & $0$ & $a$ & $c$ & $c$ & $d$ \\
$e$ & $0$ & $a$ & $b$ & $0$ & $a$ & $b$ & $c$ & $d$ & $e$ \\
$f$ & $0$ & $0$ & $0$ & $c$ & $c$ & $c$ & $f$ & $f$ & $f$ \\

$g$ & $0$ & $0$ & $a$ & $c$ & $c$ & $d$ & $f$ & $f$ & $g$ \\

$1$ & $0$ & $a$ & $b$ & $c$ & $d$ & $e$ & $f$ & $g$ & $1$
\end{tabular}

\end{center}
\end{example}

\begin{example}
\label{lrex8}
{\rm \cite[Section 15.2.1]{ior2}, \cite{ior}} Let $A=\{0,n,a,b,i,f,g,h,j,c,d,1\}$, described below. $A$ is a residuated lattice.

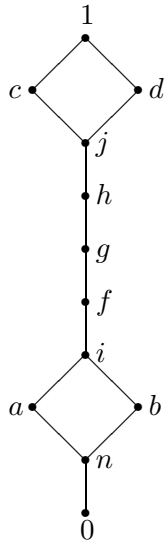
\begin{center}
\begin{picture}(70,210)(0,0)
\put(30,10){\line(0,1){20}}

\put(30,30){\line(1,1){20}}

\put(30,30){\line(-1,1){20}}
\put(30,70){\line(1,-1){20}}
\put(30,70){\line(-1,-1){20}}
\put(30,70){\line(0,1){80}}
\put(30,150){\line(1,1){20}}
\put(30,150){\line(-1,1){20}}
\put(30,190){\line(1,-1){20}}
\put(30,190){\line(-1,-1){20}}
\put(30,10){\circle*{3}}
\put(10,50){\circle*{3}}

\put(30,30){\circle*{3}}
\put(50,50){\circle*{3}}
\put(30,70){\circle*{3}}
\put(30,90){\circle*{3}}
\put(30,110){\circle*{3}}
\put(30,130){\circle*{3}}

\put(30,150){\circle*{3}}
\put(30,190){\circle*{3}}

\put(10,170){\circle*{3}}

\put(50,170){\circle*{3}}

\put(28,0){$0$}
\put(34,27){$n$}
\put(34,67){$i$}
\put(34,87){$f$}
\put(34,107){$g$}
\put(34,127){$h$}

\put(34,147){$j$}
\put(1,47){$a$}
\put(54,47){$b$}

\put(1,167){$c$}
\put(54,167){$d$}

\put(28,195){$1$}

\end{picture}

\end{center}

\begin{center}
\begin{tabular}{c|cccccccccccc}
$\rightarrow $ & $0$ & $n$ & $a$ & $b$ & $i$ & $f$ & $g$ & $h$ & $j$ & $c$ & $d$ & $1$ \\ \hline
$0$ & $1$ & $1$ & $1$ & $1$ & $1$ & $1$ & $1$ & $1$ & $1$ & $1$ & $1$ & $1$ \\
$n$ & $0$ & $1$ & $1$ & $1$ & $1$ & $1$ & $1$ & $1$ & $1$ & $1$ & $1$ & $1$ \\
$a$ & $0$ & $d$ & $1$ & $d$ & $1$ & $1$ & $1$ & $1$ & $1$ & $1$ & $1$ & $1$ \\
$b$ & $0$ & $c$ & $c$ & $1$ & $1$ & $1$ & $1$ & $1$ & $1$ & $1$ & $1$ & $1$ \\
$i$ & $0$ & $j$ & $c$ & $d$ & $1$ & $1$ & $1$ & $1$ & $1$ & $1$ & $1$ & $1$ \\
$f$ & $0$ & $h$ & $h$ & $h$ & $h$ & $1$ & $1$ & $1$ & $1$ & $1$ & $1$ & $1$ \\
$g$ & $0$ & $g$ & $g$ & $g$ & $g$ & $h$ & $1$ & $1$ & $1$ & $1$ & $1$ & $1$ \\
$h$ & $0$ & $f$ & $f$ & $f$ & $f$ & $h$ & $h$ & $1$ & $1$ & $1$ & $1$ & $1$ \\

$j$ & $0$ & $i$ & $i$ & $i$ & $i$ & $f$ & $g$ & $h$ & $1$ & $1$ & $1$ & $1$ \\
$c$ & $0$ & $b$ & $i$ & $b$ & $i$ & $f$ & $g$ & $h$ & $d$ & $1$ & $d$ & $1$ \\

$d$ & $0$ & $a$ & $a$ & $i$ & $i$ & $f$ & $g$ & $h$ & $c$ & $c$ & $1$ & $1$ \\

$1$ & $0$ & $n$ & $a$ & $b$ & $i$ & $f$ & $g$ & $h$ & $j$ & $c$ & $d$ & $1$ \end{tabular}
\end{center}

\begin{center}
\begin{tabular}{c|cccccccccccc}

$\odot $ & $0$ & $n$ & $a$ & $b$ & $i$ & $f$ & $g$ & $h$ & $j$ & $c$ & $d$ & $1$ \\ \hline

$0$ & $0$ & $0$ & $0$ & $0$ & $0$ & $0$ & $0$ & $0$ & $0$ & $0$ & $0$ & $0$ \\
$n$ & $0$ & $n$ & $n$ & $n$ & $n$ & $n$ & $n$ & $n$ & $n$ & $n$ & $n$ & $n$ \\

$a$ & $0$ & $n$ & $n$ & $n$ & $n$ & $n$ & $n$ & $n$ & $n$ & $a$ & $n$ & $a$ \\
$b$ & $0$ & $n$ & $n$ & $n$ & $n$ & $n$ & $n$ & $n$ & $n$ & $n$ & $b$ & $b$ \\
$i$ & $0$ & $n$ & $n$ & $n$ & $n$ & $n$ & $n$ & $n$ & $n$ & $a$ & $b$ & $i$ \\
$f$ & $0$ & $n$ & $n$ & $n$ & $n$ & $n$ & $n$ & $n$ & $f$ & $f$ & $f$ & $f$ \\

$g$ & $0$ & $n$ & $n$ & $n$ & $n$ & $n$ & $n$ & $f$ & $g$ & $g$ & $g$ & $g$ \\
$h$ & $0$ & $n$ & $n$ & $n$ & $n$ & $n$ & $f$ & $f$ & $h$ & $h$ & $h$ & $h$ \\
$j$ & $0$ & $n$ & $n$ & $n$ & $n$ & $f$ & $g$ & $h$ & $j$ & $j$ & $j$ & $j$ \\

$c$ & $0$ & $n$ & $a$ & $n$ & $a$ & $f$ & $g$ & $h$ & $j$ & $c$ & $j$ & $c$ \\
$d$ & $0$ & $n$ & $n$ & $b$ & $b$ & $f$ & $g$ & $h$ & $j$ & $j$ & $d$ & $d$ \\
$1$ & $0$ & $n$ & $a$ & $b$ & $i$ & $f$ & $g$ & $h$ & $j$ & $c$ & $d$ & $1$ \end{tabular}
\end{center}
\end{example}

For any residuated lattice $A$ and any $a,b\in A$, we denote $a\leftrightarrow b=(a\rightarrow b)\wedge (b\rightarrow a)$ \index{$\leftrightarrow $} and $\neg \, a=a\rightarrow 0$.\index{$\neg $}

Let $A$ be a residuated lattice, $a\in A$ and $n\in \N ^{*}$. We shall denote by $a^{n}$ \index{$a^{n}$} the following element of $A$: $\underbrace{a\odot \ldots \odot a}_{n\ {\rm of}\ a}$. We also denote $a^{0}=1$.

\begin{lemma}{\rm \cite{kow}, \cite{haj}, \cite{pic}, \cite{tur}} Let $A$ be a residuated lattice and $a,b,c\in A$. Then:
\begin{enumerate}
\item\label{6.(ii)} if $a\vee b=1$ then $a\odot b=a\wedge b$;
\item\label{6.(iii)} $(a\vee b)\odot (a\vee c)\leq a\vee (b\odot c)$, hence $(a\vee c)\odot (b\vee c)\leq (a\odot b)\vee c$ and, for any $n,k\in \N ^{*}$, $(a\vee b)^{nk}\leq a^{n}\vee b^{k}$;
\item\label{art5(iv)} $a\leq b$ iff $a\rightarrow b=1$, and $a=b$ iff $a\leftrightarrow b=1$.
\end{enumerate}
\label{calcul}
\end{lemma}

\begin{definition}
Let $L$ be a distributive lattice with $0$. An element $l$ of $L$ is said to be {\em pseudocomplemented}\index{element!pseudocomplemented} iff there exists a greatest element $m$ of $L$ which satisfies: $l\wedge m=0$; such an element $m$ is denoted $l^{*}$\index{$l^{*}$} and called the {\em pseudocomplement of $l$}.\index{pseudocomplement} $L$ is said to be {\em pseudocomplemented}\index{lattice!pseudocomplemented} iff all its elements are pseudocomplemented.
\end{definition}

We recall that, if it exists, the complement of an element of a bounded lattice coincides with its pseudocomplement.\index{pseudocomplement}

\begin{definition}
Let $L$ be a lattice. A nonempty subset $F$ of $L$ is called a {\em filter of $L$} \index{filter!of a lattice} iff it satisfies the following conditions:

\noindent (i) for all $l,m\in F$, $l\wedge m\in F$;

\noindent (ii) for all $l\in F$ and all $m\in L$, if $l\leq m$ then $m\in F$.

The set of all filters of $L$ is denoted ${\cal{F}}(L)$.\index{${\cal{F}}(L)$}

A filter $F$ of $L$ is said to be {\em proper} iff $F\neq L$.

A proper filter $P$ of $L$ is called a {\em prime filter} \index{filter!prime} iff, for all $l,m\in L$, if $l\vee m\in P$, then $l\in P$ or $m\in P$. The set of all prime filters of $L$ is called {\em the (prime) spectrum of $L$}.
\label{`filterl`}
\end{definition}

\begin{definition}
Let $A$ be a residuated lattice. A nonempty subset $F$ of $A$ is called a {\em filter of $A$} \index{filter!of a residuated lattice} iff it satisfies the following conditions:

\noindent (i) for all $a,b\in F$, $a\odot b\in F$;

\noindent (ii) for all $a\in F$ and all $b\in A$, if $a\leq b$ then $b\in F$.

The set of all filters of $A$ is denoted ${\cal{F}}(A)$.\index{${\cal{F}}(A)$}

A filter $F$ of $A$ is said to be {\em proper} iff $F\neq A$.

A proper filter $P$ of $A$ is called a {\em prime filter} \index{filter!prime} iff, for all $a,b\in A$, if $a\vee b\in P$, then $a\in P$ or $b\in P$. The set of all prime filters of $A$ is called {\em the (prime) spectrum of $A$}.
\label{`filtera`}
\end{definition}

\begin{remark}{\rm \cite{eu1}} Let $A$ be a residuated lattice, $F$ a filter of $A$ and $a,b\in A$. Then: $a\odot b\in F$ iff $a\wedge b\in F$ iff $a,b\in F$.
\label{inF}
\end{remark}
\begin{proof}
By Definition \ref{`filtera`}.\end{proof}

For all elements $x$ and all subsets $X$ of a lattice or residuated lattice $A$, we denote by $<x>$ the principal filter of $A$ generated by $x$ and by $<X>$ the filter of $A$ generated by $X$.

\begin{lemma}{\rm \cite{eu1}} Let $A$ be a residuated lattice and $a\in A$. Then $<a>=\{b\in A|(\exists \, n\in \N ^{*})\, a^{n}\leq b\}$.\index{filter!principal!of a residuated lattice}\index{$<a>$}
\label{principal}
\end{lemma}

\begin{definition}
Let $L$ be a lattice and $F$ a filter of $L$. For all $l,m\in L$, we denote $l\equiv m\ ({\rm mod}\ F)$ and say that {\em $l$ and $m$ are congruent modulo $F$} iff there exists an element $e\in F$ such that $l\wedge e=m\wedge e$. Obviously, $\equiv \ ({\rm mod}\ F)$ is a congruence relation on $L$. The quotient lattice\index{lattice!quotient} with respect to the congruence relation $\equiv \ ({\rm mod}\ F)$ is denoted $L/F$\index{$L/F$} and its elements are denoted $l/F$,\index{$l/F$} $l\in L$.
\end{definition}

\begin{definition}
Let $A$ be a residuated lattice and $F$ a filter of $A$. For all $a,b\in A$, we denote $a\equiv b\ ({\rm mod}\ F)$ and say that {\em $a$ and $b$ are congruent modulo $F$} iff $a\leftrightarrow b\in F$. Obviously, $\equiv \ ({\rm mod}\ F)$ is a congruence relation on $A$. Residuated lattices form an equational class, which ensures us that the quotient set with respect to the congruence relation $\equiv ({\rm mod}\ F)$ is a residuated lattice.\index{residuated lattice!quotient} It is denoted $A/F$\index{$A/F$} and its elements are denoted $a/F$,\index{$a/F$} $a\in A$.
\end{definition}

\blem \label{prop-cong} {\rm \cite{haj}} Let $F$ be a filter of $A$ and $a,b\in A$. Then:
\be
\item\label{aF=0Fsau1F} $a/F=1/F$ iff $a\in F$;
\item\label{aF-leq-bF} $a/F\leq b/F$ iff $a\ra b\in F$; consequently, if $a\leq b$ then $a/F\leq b/F$.
\ee
\elem

\begin{notation}
Let $A$ be a lattice (residuated lattice). For all filters $F$, $G$ of $A$, we denote $<F\cup G>$ by $F\vee G$. More generally, for any family $\{F_{t}|t\in T\}$ of filters of $A$, we denote $<\displaystyle \bigcup _{t\in T}F_{t}>$ by $\displaystyle \bigvee _{t\in T}F_{t}$.\index{$\vee $}
\label{`veefamfiltre`}
\end{notation}

\begin{proposition}{\rm \cite{cre}} Let $A$ be a lattice (residuated lattice). Then $({\cal{F}}(A),$\linebreak $\vee ,\cap ,\{1\},A)$\index{${\cal{F}}(A)$}\index{${\cal{F}}(L)$} \index{lattice!of filters} is a complete distributive lattice, whose order relation is $\subseteq $.
\label{latoffilt}
\end{proposition}

If $A$ is a bounded lattice or a residuated lattice, then the set of the complemented elements of $A$ is called the {\em Boolean center of $A$} and is denoted by $B(A)$. It is known that, for $A$ a bounded distributive lattice or a residuated lattice, this subset of $A$ is a Boolean algebra with the operations induced by those of $A$.

\begin{lemma}{\rm \cite{BusPic06},\cite{gal},\cite{kuh},\cite{eu3}} Let $A$ be a residuated lattice. Then, for every $e,f\in B(A)$ and $a\in A$, we have:

\begin{enumerate}
\item\label{e-idempotent} $e\odot e=e$;
\item\label{B(A)=} $a\in B(A)$ iff $a\vee \neg \, a=1$;
\item\label{evx} $\neg \, e\rightarrow a=e\vee a$;
\item\label{BA-e-f} $e\odot f=e\si f\in B(A)$, $e\ra f=\neg\, e\sau f\in B(A)$ and $e\leftrightarrow f=(e\ra f)\si (f\ra e)\in B(A)$.
\end{enumerate}
\label{bcomut}
\end{lemma}

\begin{proposition}{\rm \cite{cre}} Let $A$ be a distributive lattice or a residuated lattice. Then, for all $a,b\in A$: $<a>\cap <b>=<a\vee b>$.\index{filter!principal}
\label{`sisau`}
\end{proposition}

Let $A$ be a bounded distributive lattice or a residuated lattice; the definitions we are about to give are valid for both types of structures. For any non-empty subset $X$ of $A$, the {\em co-annihilator of $X$}\index{co-annihilator} is the set $X^{\top }=\{a\in A|(\forall x\in X)a\vee x=1\}$.\index{$X^{\top }$} In the case when $X$ consists of a single element $x$, we denote the co-annihilator of $X$ by $x^{\top }$\index{$x^{\top }$} and call it the {\em co-annihilator of $x$}. Also, we will denote $X^{\top \top}=(X^{\top })^{\top }$\index{$X^{\top \top}$} and $x^{\top \top}=(x^{\top })^{\top }$.\index{$x^{\top \top}$}

Notice that, for all bounded distributive lattices or residuated lattices $A$ and for all non-empty subsets $X$, $Y$ of $A$, if $X\subseteq Y$ then $Y^{\top }\subseteq X^{\top }$.

\begin{proposition}
Let $A$ be a bounded distributive lattice or a residuated lattice. Then, for any $X\subseteq A$, $X^{\top }$ is a filter of $A$.\index{$X^{\top }$}\index{co-annihilator}
\label{topfiltru}
\end{proposition}

\begin{proof}

By Remark \ref{calcul}, (\ref{6.(iii)}), for $A$ a residuated lattice and by the distributivity for $A$ a bounded distributive lattice.\end{proof}

\begin{definition}
Let $A$ be a bounded distributive lattice or a residuated lattice. Then $A$ is said to be {\em co-Stone}\index{lattice!co-Stone}\index{residuated lattice!co-Stone} (respectively {\em strongly co-Stone})\index{lattice!strongly co-Stone}\index{residuated lattice!strongly co-Stone}  iff, for all $x\in A$ (respectively all $X\subseteq A$), there exists an element $e\in B(A)$ such that $x^{\top }=<e>$ (respectively $X^{\top }=<e>$).
\label{co-Stone}
\end{definition}

Obviously, any complete co-Stone lattice (residuated lattice) is strongly co-Stone, as is shown by Proposition \ref{`sisau`} and the fact that, with the notations in the previous definition, $\displaystyle X^{\top }=\bigcap _{x\in X}x^{\top }$.

We have chosen the previous definition of co-Stone residuated lattices over the definition from \cite{rcig} for a reason that is explained by Remark \ref{nucdefco-Stone}. In \cite{rcig}, the author defines a Stone residuated lattice to be a residuated lattice $A$ that satisfies the equation: $\neg \, a\vee \neg \, \neg \, a=1$ for all $a\in A$.

For any bounded distributive lattice or residuated lattice $A$, we shall denote ${\rm CoAnn}(A)=\{X^{\top }|X\subseteq A\}$\index{${\rm CoAnn}(A)$} and, for all $F,G\in {\rm CoAnn}(A)$, we shall denote $F\vee ^{\top }G=(F^{\top }\cap G^{\top })^{\top }$.\index{$\vee ^{\top }$} More generally, for all $\{F_{t}|t\in T\}\subseteq {\rm CoAnn}(A)$, we denote $\displaystyle \bigvee _{t\in T}{}^{\top }F_{t}=\left( \bigcap _{t\in T}F_{t}^{\top }\right) ^{\top }$.

\begin{proposition}
Let $A$ be a bounded distributive lattice or a residuated lattice. Then $({\rm CoAnn}(A),\vee ^{\top },\cap ,^{\top },\{1\},A)$\index{${\rm CoAnn}(A)$} is a complete Boolean algebra.
\label{CoAnnbool}
\end{proposition}
\begin{proof}

This result can be found in \cite{leo} for BL-algebras. Its proof is also valid for bounded distributive lattices and residuated lattices.\end{proof}

\begin{definition}
Let $m$ be an infinite cardinal. An $m$-complete lattice\index{lattice!$m$-complete} is a lattice $L$ with the property that any subset $X$ of $L$ with $|X|\leq m$ has an infimum and a supremum in $L$.
\end{definition}

\begin{theorem} Let $L$ be a bounded distributive lattice and $m$ an infinite cardinal. Then the following are equivalent:

\noindent (i) for each non-empty subset $X$ of $L$ with $|X|\leq m$, there exists an element $e\in B(L)$ such that $X^{\top }=<e>$;

\noindent (ii) $L$ is a co-Stone lattice and $B(L)$ is an $m$-complete Boolean algebra;

\noindent (iii) $L_{\top \top }=\{l^{\top \top }|l\in L\}$ is an $m$-complete Boolean sublattice of ${\cal{F}}(L)$;

\noindent (iv) for all $l,p\in L$, $(l\vee p)^{\top }=l^{\top }\vee p^{\top }$ and, for each non-empty subset $X$ of $L$ with $|X|\leq m$, there exists an element $x\in L$ such that $X^{\top \top}=x^{\top }$;

\noindent (v) for each non-empty subset $X$ of $L$ with $|X|\leq m$, $X^{\top }\vee X^{\top \top }=L$.
\label{caractlco-Stone}
\end{theorem}
\begin{proof}
By duality, from \cite[Theorem 1]{dav}.\end{proof}

A bounded distributive lattice will be called an {\em $m$-co-Stone lattice}\index{lattice!$m$-co-Stone} iff the conditions of Theorem \ref{caractlco-Stone} hold for it.

\begin{definition}
Let $A$ be a bounded lattice (residuated lattice) and $B$ a subalgebra of $A$. We say that $B$ is {\em co-dense in $A$}\index{co-dense} iff, for all $a\in A\setminus \{1\}$, there exists $b\in B$ such that $a\leq b<1$ (that is $a\leq b\leq 1$ and $b\neq 1$).
\end{definition}

We denote by ${\cal{RL}}$ the category of residuated lattices and by ${\cal{D}}01$ the category of bounded distributive lattices.

For the definitions related to the inductive limit, that we present below, we are using the terminology of \cite{bus}.

A partially ordered set $(I,\leq )$ is called a {\em directed set}\index{directed set} iff, for any $i,j\in I$, there exists an element $k\in I$ such that $i\leq k$ and $j\leq k$.

\begin{definition}
Let $(I,\leq )$ be a directed set and $\cal{C}$ a category. By an {\em inductive system}\index{inductive system} of objects in $\cal{C}$ with respect to the directed index set $I$ we mean a pair $((A_{i})_{i\in I},(\phi _{ij})_{\stackrel{\scriptstyle i,j\in I}{\scriptstyle i\leq j}})$ with $(A_{i})_{i\in I}$ a family of objects of $\cal{C}$ and, for all $i,j\in I$ with $i\leq j$, $\phi _{ij}:A_{i}\rightarrow A_{j}$ a morphism in ${\cal{C}}$, such that:

\noindent (i) for every $i\in I$, $\phi _{i\, i}=1_{A_{i}}$;

\noindent (ii) for any $i,j,k\in I$ with $i\leq j\leq k$, $\phi _{jk}\circ \phi _{ij}=\phi _{ik}$.

If there is no danger of confusion, an inductive system like above will be denoted $(A_{i},\phi _{ij})$.

\end{definition}

\begin{definition}
Let $(A_{i},\phi _{ij})$ be an inductive system of objects in a ca\-te\-go\-ry ${\cal{C}}$ relative to a directed index set $I$. A pair $(A,(\phi _{i})_{i\in I})$, with $A$ an object in ${\cal{C}}$ and, for all $i\in I$, $\phi _{i}:A_{i}\rightarrow A$ a morphism in ${\cal{C}}$, is called {\em inductive limit} \index{inductive limit} of the inductive system $(A_{i},\phi _{ij})$ iff:

\noindent (i) for every $i,j\in I$ with $i\leq j$, $\phi _{j}\circ \phi _{ij}=\phi _{i}$;

\begin{center}
\begin{picture}(60,60)(0,0)
\put(7,37){$A_{i}$}

\put(20,40){\vector(1,0){20}}
\put(42,37){$A_{j}$}
\put(22,44){$\phi _{ij}$}
\put(45,35){\vector(0,-1){20}}
\put(47,22){$\phi _{j}$}
\put(42,5){$A$}
\put(18,35){\vector(1,-1){22}}
\put(16,22){$\phi _{i}$}
\end{picture}

\end{center}

\noindent (ii) for any object $B$ of ${\cal{C}}$ and any family $(f_{i})_{i\in I}$ of morphisms in ${\cal{C}}$ such that, for all $i\in I$, $f_{i}:A_{i}\rightarrow B$ and, for all $i,j\in I$ with $i\leq j$, $f_{j}\circ \phi _{ij}=f_{i}$, there is a unique morphism $f:A\rightarrow B$ in ${\cal{C}}$ such that, for every $i\in I$, $f\circ \phi _{i}=f_{i}$.

\begin{center}
\begin{picture}(60,60)(0,0)
\put(7,37){$A_{i}$}

\put(20,40){\vector(1,0){20}}
\put(42,37){$A$}
\put(22,44){$\phi _{i}$}

\put(45,35){\vector(0,-1){20}}

\put(47,22){$f$}
\put(42,5){$B$}
\put(18,35){\vector(1,-1){22}}

\put(16,22){$f_{i}$}
\end{picture}
\end{center}
\end{definition}

It is immediate that the inductive limit\index{inductive limit} of a given inductive system is unique up to an isomorphism, that is, if $(A,(\phi _{i})_{i\in I})$ and $(B,(\psi _{i})_{i\in I})$ are two inductive limits of the same inductive system, then there exists a unique isomorphism $f:A\rightarrow B$ such that, for every $i\in I$, $f\circ \phi _{i}=\psi _{i}$.

The next lemma is known and easy to prove.

\begin{lemma}
Let $((A_{i})_{i\in I},(\phi _{ij})_{\stackrel{\scriptstyle i,j\in I}{\scriptstyle i\leq j}})$ and $((B_{i})_{i\in I},(\psi _{ij})_{\stackrel{\scriptstyle i,j\in I}{\scriptstyle i\leq j}})$ be two inductive systems in the same category, with inductive limits $(A,(\phi _{i})_{i\in I})$ and $(B,(\psi _{i})_{i\in I})$, respectively. If, for every $i\in I$, there exists an isomorphism $f_{i}:A_{i}\rightarrow B_{i}$ such that, for all $i,j\in I$ with $i\leq j$, $\psi _{ij}\circ f_{i}=f_{j}\circ \phi _{ij}$, then there exists an isomorphism $f:A\rightarrow B$ such that, for all $i\in I$, $f\circ \phi _{i}=\psi _{i}\circ f_{i}$.
\label{izomlimind}
\end{lemma}

We say that a category ${\cal{C}}$ is a {\em category with inductive limits}\index{category with inductive limits} iff every inductive system in ${\cal{C}}$ has an inductive limit. The category of sets, the category of residuated lattices and the category of bounded distributive lattices are categories with inductive limits. Indeed, \cite[Example 4.7.2]{bus} contains the construction of the inductive limits in any equational class of algebras. 

In the following, we shall present a construction for the inductive limit in the category of residuated lattices. As we believe that this construction is known, we shall not give any proofs here.\index{inductive limit} See also \cite{bus}.

Let $(A_{i},\phi _{ij})$ be an inductive system in ${\cal{RL}}$. We denote by $\displaystyle \coprod _{i\in I}A_{i}$ the disjoint union of the family $(A_{i})_{i\in I}$. Let us consider the following relation on $\displaystyle \coprod _{i\in I}A_{i}$: for all $i,j\in I$, all $a\in A_{i}$ and all $b\in A_{j}$, $a\sim b$ iff there exists $k\in I$ such that $i\leq k$, $j\leq k$ and $\phi _{ik}(a)=\phi _{jk}(b)$. It is immediate that $\sim $ is an equivalence relation on $\displaystyle \coprod _{i\in I}A_{i}$. We denote by $A$ the quotient set $\left( \displaystyle \coprod _{i\in I}A_{i}\right) /\sim$ and by $[a]$ the equivalence class of an element $a\in \displaystyle \coprod _{i\in I}A_{i}$. For any $i\in I$, let $\phi _{i}:A_{i}\rightarrow A$, for all $a\in A_{i}$, $\phi _{i}(a)=[a]$.

Let us define residuated lattice operations on $A$. We define $0=[0]$ and $1=[1]$. Obviously, this definition does not depend on the residuated lattice $A_{i}$ the 0 and the 1 are taken from. Let $[a],[b]\in A$. Let $i,j\in I$ such that $a\in A_{i}$ and $b\in A_{j}$. Then, by the definition of the directed set, there exists $k\in I$ such that $i\leq k$ and $j\leq k$. We define $[a]\vee [b]=[\phi _{ik}(a)\vee \phi _{jk}(b)]$ and $[a]\wedge [b]=[\phi _{ik}(a)\wedge \phi _{jk}(b)]$. The same for $\odot $ and $\rightarrow $. Here is the definition of the partial order relation: for all $[a],[b]\in A$ with $a\in A_{i}$ and $b\in A_{j}$ for some $i,j\in I$, we define: $[a]\leq [b]$ iff there exists $k\in I$ such that $i\leq k$, $j\leq k$ and $\phi _{ik}(a)\leq \phi _{jk}(b)$.

Then $(A,(\phi _{i})_{i\in I})$ is an inductive limit of the inductive system $(A_{i},\phi _{ij})$ in the category $\cal{RL}$.

A similar construction can be done for inductive limits in the category ${\cal{D}}01$.\index{inductive limit}

Now let $P(B)$\index{$P(B)$} be the set of the finite partitions\index{poset of finite partitions} of a Boolean algebra $B$, that is $\displaystyle P(B)=\{\{x_{1},\ldots ,x_{n}\}|n\in \N ^{*},x_{1},\ldots ,x_{n}\in B\setminus \{0\},\bigvee _{i=1}^{n}x_{i}=1, (\forall i,j\in \overline{1,n})i\neq j\Rightarrow x_{i}\wedge x_{j}=0\}$. We define the partial order $\leq $ on $P(B)$ by: for all $p,q\in P(B)$, $p\leq q$ iff $q$ is a refinement of $p$, that is: $p=\{x_{1},\ldots ,x_{n}\}$ and $q=\{y_{ij}|i\in \overline{1,n},(\forall i\in \overline{1,n})j\in \overline{1,k_{i}}\}$, where $n,k_{1},\ldots ,k_{n}\in \N ^{*}$ and, for all $i\in \overline{1,n}$, $\displaystyle \bigvee _{j=1}^{k_{i}}y_{ij}=x_{i}$. For all $p,q\in P(B)$ with $p\leq q$, we define $k_{pq}:q\rightarrow p$, for all $a\in q$ and $b\in p$, $k_{pq}(a)=b$ iff $a\leq b$ (it is easily seen that, for every $a\in q$, there exists a unique $b\in p$ such that $a\leq b$; namely, with the notations above for the elements of $p$ and those of $q$, for all $i\in \overline{1,n}$ and all $j\in \overline{1,k_{i}}$, $k_{pq}(y_{ij})=x_{i}$). The fact that the functions $k_{pq}$ are well defined is obvious (if, for an $a\in q$, there exist $b_{1},b_{2}\in p$, $b_{1}\neq b_{2}$ and $a\leq b_{1}$, $a\leq b_{2}$, then $a\leq b_{1}\wedge b_{2}=0$, so $a=0$, which is a contradiction to the definition of $P(B)$).

Let us now turn our attention to the reticulation of a residuated lattice. The reticulation of an algebra was first defined by Simmons (\cite{sim}) for commutative rings and then by Belluce for MV-algebras (\cite{bel1}). Later, it was extended by Belluce to non-commutative rings (\cite{bel2}) and then it was defined for quantales (\cite{geo1}) and for BL-algebras (\cite{leo1}, \cite{leo}). In each of the papers cited above, although it is not explicitely defined this way, the reticulation of an algebra $A$ is a pair $(L(A),\lambda )$ consisting of a bounded distributive lattice $L(A)$ and a surjection $\lambda :A\rightarrow L(A)$ such that the function given by the inverse image of $\lambda $ induces (by restriction) a homeomorphism of topological spaces between the prime spectrum of $L(A)$ and that of $A$. This construction allows many properties to be transferred between $L(A)$ and $A$, and this transfer of properties between the category of bounded distributive lattices and another category (in our case that of residuated lattices) is the very purpose of the reticulation.

Here is the definition that we gave in \cite{eu1} for the reticulation of a residuated lattice. This axiomatic definition is purely algebraic, thus being an innovation in the study of the reticulation, as in previous work the reticulation of an algebra was defined by its construction.

\begin{definition}{\rm \cite{eu1}} Let $A$ be a residuated lattice. A {\em reticulation of $A$} is a pair $(L,\lambda )$, where $L$ is a bounded distributive lattice and $\lambda :A\rightarrow L$ is a function that satisfies conditions 1)-5) below:

\noindent 1) for all $a,b\in A$, $\lambda (a\odot b)=\lambda (a)\wedge \lambda (b)$;

\noindent 2) for all $a,b\in A$, $\lambda (a\vee b)=\lambda (a)\vee \lambda (b)$;

\noindent 3) $\lambda (0)=0$; $\lambda (1)=1$;

\noindent 4) $\lambda $ is surjective;

\noindent 5) for all $a,b\in A$, $\lambda (a)\leq \lambda (b)$ iff $(\exists \, n\in \N ^{*})\, a^{n}\leq b$.
\label{reticulatia}
\end{definition}

In \cite{eu1} and \cite{eu2} we proved that this definition is in accordance with the general notion of reticulation applied to residuated lattices, more precisely that, given a residuated lattice $A$ and a pair $(L,\lambda )$ consisting of a bounded distributive lattice $L$ and a function $\lambda :A\rightarrow L$, we have: if $\lambda $ satisfies conditions 1)-5) above, then its inverse image induces (by restriction) a homeomorphism between the prime spectrum of $L$ and that of $A$ (regarded as topological spaces with the Stone topologies); and conversely: if the function given by the inverse image of $\lambda $ takes prime filters of $L$ to prime filters of $A$ and its restriction to the prime spectrum of $L$ is a homeomorphism between the prime spectrum of $L$ and that of $A$ (with the Stone topologies), then $\lambda $ satisfies conditions 1)-5) from the definition above.

\begin{lemma}
With the notations in Definition \ref{reticulatia}, a function $\lambda $ that verifies conditions 1)-3) also satisfies:

\noindent a) $\lambda $ is order-preserving;\index{a)}

\noindent b) for all $a,b\in A$, $\lambda (a\wedge b)=\lambda (a)\wedge \lambda (b)$;\index{b)}

\noindent c) for all $a\in A$ and all $n\in \N ^{*}$, $\lambda (a^{n})=\lambda (a)$.\index{c)}
\label{`abc`}
\end{lemma}

We shall use the notations of the conditions 1)-5) and of the properties a)-c) in what follows. 

The following theorem states the existence and uniqueness of the re\-ti\-cu\-la\-tion for any residuated lattice. Here, as in many other cases, one can see the usefulness of the axiomatic purely algebraic definition of the reticulation, which allows us to provide a simple algebraic proof for the uniqueness of the reticulation; this is another important novelty, as in previous work the argument for the uniqueness of the reticulation was of topological nature and consisted of the fact that there is at most one bounded distributive lattice whose prime spectrum is homeomorphic to a given topological space.

\begin{theorem}{\rm \cite{eu1}} Let $A$ be a residuated lattice. Then there exists a reticulation of $A$. Let $(L_{1},\lambda _{1})$, $(L_{2},\lambda _{2})$ be two reticulations of $A$. Then there exists an isomorphism of bounded lattices $f:L_{1}\rightarrow L_{2}$ such that $f\circ \lambda _{1}=\lambda _{2}$.
\label{`unicitatea`}
\end{theorem}

Until mentioned otherwise, let $A$ be a residuated lattice and $(L,\lambda )$ its reticulation.

\begin{lemma}{\rm \cite{eu1}} For any filters $F$, $G$ of $A$, we have: $\lambda (F)=\lambda (G)$ iff $F=G$.
\label{lambdaegal}
\end{lemma}

\begin{remark}{\rm \cite{eu1}} For all $a\in A$, $\lambda (<a>)=<\lambda (a)>$.
\label{lambdagen}
\end{remark}

\begin{lemma}{\rm \cite{eu1}} For any filter $F$ of $A$, $\lambda (F)$ is a filter of $L$.
\label{`lambdaf`}
\end{lemma}

For any filter $F$ of $A$, let us denote $\mu (F)=\{\lambda (a)|a\in F\}=\lambda (F)$. By Lemma \ref{`lambdaf`}, we have defined a function $\mu :{\cal{F}}(A)\rightarrow {\cal{F}}(L)$.\index{$\mu $}

\begin{proposition}
The function $\mu :{\cal{F}}(A)\rightarrow {\cal{F}}(L)$ defined above is a bounded lattice isomorphism.\index{$\mu $}
\label{`izomf`}
\end{proposition}

In \cite{eu1} and \cite{eu2}, we defined {\em the reticulation functor} ${\cal{L}}:{\cal{RL}}\rightarrow {\cal{D}}01$. If $A$ is a residuated lattice and $(L(A),\lambda _{A})$ is its reticulation, then ${\cal{L}}(A)=L(A)$. If $B$ is another residuated lattice, $(L(B),\lambda _{B})$ is its reticulation and $f:A\rightarrow B$ is a morphism of residuated lattices, then ${\cal{L}}(f):{\cal{L}}(A)=L(A)\rightarrow {\cal{L}}(B)=L(B)$, for all $a\in A$, ${\cal{L}}(f)(\lambda _{A}(a))=\lambda _{B}(f(a))$. This definition makes ${\cal{L}}$ a covariant functor from ${\cal{RL}}$ to ${\cal{D}}01$.

Here is an alternate definition of ${\cal{L}}$, which is in accordance with the one above:

\begin{proposition}
Let $A,B$ be residuated lattices, $f:A\rightarrow B$ a morphism of residuated lattices and $({\cal{L}}(A),\lambda _{A}),({\cal{L}}(B),\lambda _{B})$ the reticulations of $A$ and $B$, respectively. Then there exists a unique bounded lattice morphism $h:{\cal{L}}(A)\rightarrow {\cal{L}}(B)$ such that $h\circ \lambda _{A}=\lambda _{B}\circ f$ (i. e. that makes the diagram below commutative).

\begin{center}
\begin{picture}(170,70)(0,0)
\put(17,53){$A$}
\put(20,50){\vector(0,-1){20}}
\put(4,38){$\lambda _{A}$}

\put(125,38){$\lambda _{B}$}
\put(20,18){${\cal{L}}(A)$}
\put(26,58){\vector(1,0){90}}
\put(68,61){$f$}
\put(119,53){$B$}
\put(119,18){${\cal{L}}(B)$}
\put(47,22){\vector(1,0){69}}
\put(122,50){\vector(0,-1){20}}

\put(60,9){$h={\cal{L}}(f)$}
\end{picture}
\end{center}
\label{`lmorfisme`}
\end{proposition}

\begin{definition}
With the notations in Proposition \ref{`lmorfisme`}, set ${\cal{L}}(f)=h$.\index{functor!${\cal{L}}$}\index{functor!reticulation functor}
\label{deflmorfisme}
\end{definition}

Here is the second construction of the reticulation from \cite{eu1}: let $A$ be a residuated lattice, ${\cal{PF}}(A)$ be the set of the principal filters of $A$ and $\lambda :A\rightarrow {\cal{PF}}(A)$ the function given by: for all $a\in A$, $\lambda (a)=<a>$.

\begin{theorem}{\rm \cite{eu1}} $(({\cal{PF}}(A),\cap ,\vee ,A,\{1\}),\lambda )$ is a reticulation of $A$.\index{reticulation!existence}
\label{`alternate`}
\end{theorem}

Notice that the partial order relation of the lattice $({\cal{PF}}(A),\cap ,\vee ,A,\{1\})$ is $\supseteq $.

Here is an example of reticulation of a residuated lattice that we will use in the sequel:

\begin{example}
\label{ex3}
{\rm \cite{eu1}} Let $A$ be the residuated lattice in Example \ref{lrex3}. The principal filters of this residuated lattice are: $<0>=<b>=A$, $<a>=<c>=\{a,c,1\}$, $<d>=\{d,1\}$, $<1>=\{1\}$, so ${\cal{L}}(A)=$\linebreak $\{<0>,<a>,<d>,<1>\}$, with the following lattice structure:

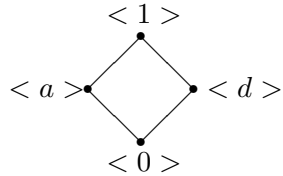
\begin{center}
\begin{picture}(100,70)(0,0)
\put(50,11){\circle*{3}}
\put(30,31){\circle*{3}}
\put(70,31){\circle*{3}}
\put(50,51){\circle*{3}}
\put(50,11){\line(1,1){20}}
\put(50,11){\line(-1,1){20}}

\put(70,31){\line(-1,1){20}}
\put(30,31){\line(1,1){20}}

\put(37,0){$<0>$}
\put(0,28){$<a>$}
\put(75,28){$<d>$}
\put(37,56){$<1>$}

\end{picture}

\end{center}

\end{example}

Here are three preservation properties of the reticulation functor for residuated lattices. 

\begin{proposition}{\rm \cite{eu5}} ${\cal{L}}$ preserves finite direct products. More precisely, if $n\in \N ^{*}$, $A_{1},A_{2},\ldots A_{n}$ are residuated lattices, $\displaystyle A=\prod _{i=1}^{n}A_{i}$, for each $i\in \overline{1,n}$, $({\cal{L}}(A_{i}),\lambda _{i})$ is a reticulation of $A_{i}$, and $\displaystyle \lambda :A\rightarrow \prod _{i=1}^{n}{\cal{L}}(A_{i})$, for all $(a_{1},\ldots ,a_{n})\in A$, $\lambda (a_{1},\ldots ,a_{n})=(\lambda _{1}(a_{1}),\ldots ,\lambda _{n}(a_{n}))$, then $\displaystyle (\prod _{i=1}^{n}{\cal{L}}(A_{i}),\lambda )$ is a reticulation of $A$.
\label{findirprod}
\end{proposition}

\begin{proposition}{\rm \cite{eu5}} ${\cal{L}}$ preserves quotients.\index{lattice!quotient}\index{residuated lattice!quotient} More precisely, if $A$ is a residuated lattice, $F$ is a filter of $A$, $({\cal{L}}(A),\lambda )$ is the reticulation of $A$, $({\cal{L}}(A/F),\lambda _{1})$ is the reticulation of $A/F$ and $h:{\cal{L}}(A)/\lambda (F)\rightarrow {\cal{L}}(A/F)$, for all $a\in A$, $h(\lambda (a)/\lambda (F))=\lambda _{1}(a/F)$, then $h$ is a bounded lattice isomorphism.
\label{presquot} 
\end{proposition}

\begin{proposition}{\rm \cite{eu5}} ${\cal{L}}$ preserves inductive limits.\index{inductive limit} More precisely, if $((A_{i})_{i\in I},(\phi _{ij})_{\stackrel{\scriptstyle i,j\in I}{\scriptstyle i\leq j}})$ is an inductive system of residuated lattices and $(A,$\linebreak $(\phi _{i})_{i\in I})$ is its inductive limit, then $({\cal{L}}(A),({\cal{L}}(\phi _{i}))_{i\in I})$ is the inductive limit of the inductive system $(({\cal{L}}(A_{i}))_{i\in I},({\cal{L}}(\phi _{ij}))_{\stackrel{\scriptstyle i,j\in I}{\scriptstyle i\leq j}})$.
\label{limind}
\end{proposition}

\section{Co-Stone Algebras}
\label{co-Stonealgebras}

\hspace*{11pt} This section contains several properties transferred between ${\cal{D}}01$ and ${\cal{RL}}$ through ${\cal{L}}$, related to co-Stone structures.
 
Concerning co-Stone and strongly co-Stone structures (by structure we mean here bounded distributive lattice or residuated lattice), the first question that arises is whether they exist. Naturally, any strongly co-Stone structure is co-Stone and any complete co-Stone structure is strongly co-Stone. The answer to the question above is given by the fact that the trivial structure is strongly co-Stone and, moreover, any chain is strongly co-Stone, because a chain $A$ clearly has all co-annihilators equal to $\{1\}=<1>$, except for $1^{\top }$, which is equal to $A=<0>$.

Until mentioned otherwise, let $A$ be a residuated lattice and $({\cal{L}}(A),\lambda )$ its reticulation.

\begin{lemma}
For any $a\in A$, we have: $\lambda (a)=1$ iff $a=1$, and $\lambda (a)=0$ iff there exists $n\in \N ^{*}$ such that $a^{n}=0$.
\label{unu}
\end{lemma}
\begin{proof}
By conditions 3) and 5), we get: $\lambda (a)=1$ iff $1\leq \lambda (a)$ iff $\lambda (1)\leq \lambda (a)$ iff there exists $n\in \N ^{*}$ such that $1^{n}\leq a$ iff $1\leq a$ iff $a=1$.

Again by conditions 3) and 5), we have: $\lambda (a)=0$ iff $\lambda (a)\leq 0$ iff $\lambda (a)\leq \lambda (0)$ iff there exists $n\in \N ^{*}$ such that $a^{n}\leq 0$ iff there exists $n\in \N ^{*}$ such that $a^{n}=0$.\end{proof}

\begin{lemma}
Let $a\in A$. Then:
\begin{enumerate}
\item\label{noua1} if $a\in B(A)$, then $\lambda (a)\in B({\cal{L}}(A))$;
\item\label{noua2} $\lambda (a)\in B({\cal{L}}(A))$ iff there exists an $n\in \N ^{*}$ such that $a^{n}\in B(A)$.
\end{enumerate}
\label{noua}
\end{lemma}
\begin{proof}
\noindent (\ref{noua1}) By properties 2), b) and 3).

\noindent (\ref{noua2}) If, for a certain $n\in \N ^{*}$, $a^{n}\in B(A)$, then, by (\ref{noua1}) and property c), $\lambda (a)=\lambda (a^{n})\in B({\cal{L}}(A))$.

If $\lambda (a)\in B({\cal{L}}(A))$, then, by condition 4), there exists $b\in A$ such that $\lambda (a)\vee \lambda (b)=1$ and $\lambda (a)\wedge \lambda (b)=0$. Using conditions 2) and 1), Lemma \ref{unu} and Lemma \ref{calcul}, (\ref{6.(iii)}) and (\ref{6.(ii)}), we find that this is equivalent to $\lambda (a\vee b)=1$ and $\lambda (a\odot b)=0$, which in turn is equivalent to $a\vee b=1$ and $(a\odot b)^{n}=0$ for some $n\in \N ^{*}$, hence $a^{n}\vee b^{n}\geq 1^{n}=1$ and $a^{n}\odot b^{n}=0$, so $a^{n}\vee b^{n}=1$ and $a^{n}\odot b^{n}=0$, so $a^{n}\wedge b^{n}=0$. Hence $a^{n}\in B(A)$.\end{proof}

\begin{proposition}
$\lambda \mid _{B(A)}:B(A)\rightarrow B({\cal{L}}(A))$ is an isomorphism of Boolean algebras.
\label{izomb}
\end{proposition}
\begin{proof}
By Lemma \ref{noua}, (\ref{noua1}), for all $a\in B(A)$, $\lambda (a)\in B({\cal{L}}(A))$. Properties 2), b) and 3) imply that $\lambda $ also preserves the complement, hence it is a Boolean morphism. Let $a,b\in B(A)$ such that $\lambda (a)=\lambda (b)$. By property 5) and Lemma \ref{bcomut}, (\ref{e-idempotent}), $\lambda (a)=\lambda (b)$ iff $\lambda (a)\leq \lambda (b)$ and $\lambda (b)\leq \lambda (a)$ iff $a^{n}\leq b$ and $b^{k}\leq a$ for some $n,k\in \N ^{*}$ iff $a\leq b$ and $b\leq a$ iff $a=b$. Therefore $\lambda \mid _{B(A)}$ is injective. Let $f\in B({\cal{L}}(A))$. By condition 4), there exists $a\in A$ such that $\lambda (a)=f$. By Lemma \ref{noua}, (\ref{noua2}), there exists an $n\in \N ^{*}$ such that $a^{n}\in B(A)$, and $\lambda (a^{n})=\lambda (a)=f$, by property c), so $\lambda \mid _{B(A)}:B(A)\rightarrow B({\cal{L}}(A))$ is also surjective.\end{proof}

\begin{remark}
For any subset $X$ of $A$, $\lambda (X^{\top })=\lambda (X)^{\top }$.
\label{comuttop}
\end{remark}

\begin{proof}
By conditions 4) and 2) and Lemma \ref{unu}, we have: $\lambda (X)^{\top }=\{\lambda (a)|a\in A,(\forall x\in X)\lambda (a)\vee \lambda (x)=1\}=\{\lambda (a)|a\in A,(\forall x\in X)\lambda (a\vee x)=1\}=\{\lambda (a)|a\in A,(\forall x\in X)a\vee x=1\}=\lambda (X^{\top })$.\end{proof}

\begin{proposition}
$A$ is a co-Stone residuated lattice\index{residuated lattice!co-Stone} iff ${\cal{L}}(A)$ is a co-Stone lattice.\index{lattice!co-Stone}
\label{co-Stoneiff}

\end{proposition}

\begin{proof}
Assume that $A$ is a co-Stone residuated lattice and let $l\in {\cal{L}}(A)$. $\lambda $ is surjective, hence there exists $a\in A$ with $\lambda (a)=l$. By Definition \ref{co-Stone}, there exists $e\in B(A)$ such that $a^{\top }=<e>$. By Lemma \ref{noua}, (\ref{noua1}), $\lambda (e)\in B({\cal{L}}(A))$. By Remarks \ref{comuttop} and \ref{lambdagen}, $l^{\top }=\lambda (a)^{\top }=\lambda (a^{\top })=\lambda (<e>)=<\lambda (e)>$. Therefore ${\cal{L}}(A)$ is a co-Stone lattice.

Now conversely: assume that ${\cal{L}}(A)$ is a co-Stone lattice and let $a\in A$. By Definition \ref{co-Stone}, the surjectivity of $\lambda $ and Remark \ref{comuttop}, there exists $e\in A$, such that $\lambda (e)\in B({\cal{L}}(A))$ and $\lambda (a^{\top })=\lambda (a)^{\top }=<\lambda (e)>$. By Lemma \ref{noua}, (\ref{noua2}), there exists an $n\in \N ^{*}$ such that $e^{n}\in B(A)$. By property c) and Remark \ref{lambdagen}, $\lambda (a^{\top })=<\lambda (e)>=<\lambda (e^{n})>=\lambda (<e^{n}>)$. By Proposition \ref{topfiltru} and Lemma \ref{lambdaegal}, we get $a^{\top }=<e^{n}>$. So $A$ is a co-Stone residuated lattice.\end{proof}

\begin{proposition}
$A$ is a strongly co-Stone residuated lattice\index{residuated lattice!strongly co-Stone}  iff ${\cal{L}}(A)$ is a strongly co-Stone lattice.\index{lattice!strongly co-Stone}
\label{stronglyco-Stoneiff}
\end{proposition}

\begin{proof}
Similar to the proof of Proposition \ref{co-Stoneiff}.\end{proof}

\begin{proposition}
Let $A$ be a residuated lattice. Then ${\rm CoAnn}(A)$\index{${\rm CoAnn}(A)$} and ${\rm CoAnn}({\cal{L}}(A))$ are isomorphic Boolean algebras.
\label{izommu}
\end{proposition}

\begin{proof}
Let $({\cal{L}}(A),\lambda )$ be the reticulation of $A$ and $\mu :{\rm CoAnn}(A)\rightarrow $\linebreak ${\rm CoAnn}({\cal{L}}(A))$, for all $F\in {\rm CoAnn}(A)$, $\mu (F)=\lambda (F)$. Proposition \ref{`izomf`} and Remark \ref{comuttop} show that $\mu $ is injective, preserves the intersection, the first and the last element and the complement of ${\rm CoAnn}(A)$, hence it is an injective morphism of Boolean algebras. For all $F\in {\rm CoAnn}({\cal{L}}(A))$, there exists $X\subseteq {\cal{L}}(A)$ such that $F=X^{\top }$. By the surjectivity of $\lambda $, there exists $Y\subseteq A$ such that $\lambda (Y)=X$. $Y^{\top }\in {\rm CoAnn}(A)$ and, by Remark \ref{comuttop}, $\mu (Y^{\top })=\lambda (Y^{\top })=\lambda (Y)^{\top }=X^{\top }=F$. So $\mu $ is also surjective, hence it is a Boolean isomorphism.\end{proof}

\begin{corollary}

With the notations in the proof of Proposition \ref{izommu}, for all $F\in {\rm CoAnn}({\cal{L}}(A))$, $\mu ^{-1}(F^{\top })=\mu ^{-1}(F)^{\top }$.
\label{izommu-1}
\end{corollary}
\begin{proof}

By Remark \ref{comuttop}.\end{proof}

\begin{theorem}

Let $A$ be a residuated lattice and $m$ an infinite cardinal. Then the following are equivalent:

\noindent (I) for each non-empty subset $X$ of $A$ with $|X|\leq m$, there exists an element $e\in B(A)$ such that $X^{\top }=<e>$;

\noindent (II) $A$ is a co-Stone residuated lattice and $B(A)$ is an $m$-complete Boolean algebra;

\noindent (III) $A_{\top \top }=\{a^{\top \top }|a\in A\}$ is an $m$-complete Boolean sublattice of ${\cal{F}}(A)$;

\noindent (IV) for all $a,b\in A$, $(a\vee b)^{\top }=a^{\top }\vee b^{\top }$ and, for each non-empty subset $X$ of $A$ with $|X|\leq m$, there exists an element $x\in A$ such that $X^{\top \top}=x^{\top }$;

\noindent (V) for each non-empty subset $X$ of $A$ with $|X|\leq m$, $X^{\top }\vee X^{\top \top }=A$.
\label{caractlrco-Stone}
\end{theorem}
\begin{proof}
Let $({\cal{L}}(A),\lambda )$ be the reticulation of $A$. Let us denote by $(i^{\prime })$, $(ii^{\prime })$, $(iii^{\prime })$, $(iv^{\prime })$, $(v^{\prime })$ the equivalents of conditions (i), (ii), (iii), (iv), respectively (v) from Theorem \ref{caractlco-Stone} for ${\cal{L}}(A)$ instead of $L$. By Theorem \ref{caractlco-Stone}, it is sufficient to prove that condition (I) is equivalent with condition $(i^{\prime })$ and the same is valid for conditions (II), (III), (IV), (V) with conditions $(ii^{\prime })$, $(iii^{\prime })$, $(iv^{\prime })$, respectively $(v^{\prime })$.

\noindent $(I)\Rightarrow (i^{\prime })$: Let $\emptyset \neq X\subseteq {\cal{L}}(A)$ with $|X|\leq m$. The fact that $\lambda $ is surjective implies that there exists $\emptyset \neq Y\subseteq A$ with $|Y|=|X|\leq m$ and $\lambda (Y)=X$. By (I), there exists $e\in B(A)$ such that $Y^{\top }=<e>$. By Lemma \ref{noua}, (\ref{noua1}), $\lambda (e)\in B({\cal{L}}(A))$. By Remarks \ref{comuttop} and \ref{lambdagen}, $X^{\top }=\lambda (Y)^{\top }=\lambda (Y^{\top })=\lambda (<e>)=<\lambda (e)>$.

\noindent $(i^{\prime })\Rightarrow (I)$: Let $\emptyset \neq X\subseteq A$ with $|X|\leq m$. Then $\lambda (X)\neq \emptyset $ and $|\lambda (X)|\leq |X|\leq m$, so there exists $f\in B({\cal{L}}(A))$ such that $\lambda (X)^{\top }=<f>$. $\lambda $ is surjective, so there exists $e\in A$ such that $\lambda (e)=f$. By Lemma \ref{noua}, (\ref{noua2}), there exists $n\in \N ^{*}$ such that $e^{n}\in B(A)$. Using Remarks \ref{comuttop} and \ref{lambdagen} and property c), we get $\lambda (X^{\top })=\lambda (X)^{\top }=<\lambda (e)>=<\lambda (e^{n})>=\lambda (<e^{n}>)$, which, by Proposition \ref{topfiltru} and Lemma \ref{lambdaegal}, implies $X^{\top }=<e^{n}>$.

Propositions \ref{co-Stoneiff} and \ref{izomb} ensure us that (II) and $(ii^{\prime })$ are equivalent.

\noindent $(III)\Leftrightarrow (iii^{\prime })$: Let us consider the posets $(A_{\top \top },\subseteq )$ and $({\cal{L}}(A)_{\top \top },\subseteq )$. By Proposition \ref{topfiltru}, $A_{\top \top }\subseteq {\cal{F}}(A)$ and ${\cal{L}}(A)_{\top \top }\subseteq {\cal{F}}({\cal{L}}(A))$. Let $\mu $ be the bounded lattice isomorphism from Proposition \ref{`izomf`} and $\psi :A_{\top \top }\rightarrow {\cal{L}}(A)_{\top \top }$ be the restriction of $\mu $ to $A_{\top \top }$, that is: for all $a\in A$, $\psi (a^{\top \top })=\mu (a^{\top \top })=\lambda (a^{\top \top })=\lambda (a)^{\top \top}$, where the last equality was obtained from Remark \ref{comuttop} and shows that $\psi $ is well defined. Propositions \ref{topfiltru} and \ref{`izomf`} imply that $\psi $ is an injective order-preserving function, and the fact that $\lambda $ is surjective implies that $\psi $ is surjective; this and Proposition \ref{`izomf`} show that $\psi $ is an order isomorphism. Therefore $A_{\top \top }$ is an $m$-complete Boolean algebra iff ${\cal{L}}(A)_{\top \top }$ is an $m$-complete Boolean algebra. 

\noindent $(IV)\Rightarrow (iv^{\prime })$: Let $a,b\in A$. We will use the surjectivity of $\lambda $. By condition 2), Remark \ref{comuttop} and Proposition \ref{`izomf`}, $(\lambda(a)\vee \lambda (b))^{\top }=\lambda (a\vee b)^{\top }=\lambda ((a\vee b)^{\top })=\lambda (a^{\top }\vee b^{\top })=\lambda (a^{\top })\vee \lambda (b^{\top })=\lambda (a)^{\top }\vee \lambda (b)^{\top }$. Let $\emptyset \neq X\subseteq {\cal{L}}(A)$ with $|X|\leq m$. By the surjectivity of $\lambda $, there exists $\emptyset \neq Y\subseteq A$ with $\lambda (Y)=X$ and $|Y|=|\lambda (Y)|=|X|\leq m$ ($Y$ can be obtained by choosing, for each $x\in X$, only one $y\in A$ such that $\lambda (y)=x$ and taking $Y$ to be the set of all these elements $y$). This implies that there exists $y\in A$ such that $Y^{\top \top}=y^{\top }$, which in turn, by Remark \ref{comuttop}, implies that $X^{\top \top}=\lambda (Y)^{\top \top }=\lambda (Y^{\top \top})=\lambda (y^{\top })=\lambda (y)^{\top }$.

\noindent $(iv^{\prime })\Rightarrow (IV)$: Let $a,b\in A$. We have: $(\lambda(a)\vee \lambda (b))^{\top }=\lambda (a)^{\top }\vee \lambda (b)^{\top }$, which, by computations similar to the ones above, is equivalent to: $\lambda ((a\vee b)^{\top })=\lambda (a^{\top }\vee b^{\top })$. This, by Proposition \ref{topfiltru} and Lemma \ref{lambdaegal}, implies that $(a\vee b)^{\top }=a^{\top }\vee b^{\top }$. Let $\emptyset \neq Y\subseteq A$ with $|Y|\leq m$. Then $\lambda (Y)\neq \emptyset $ and $|\lambda (Y)|\leq |Y|\leq m$, so, by the surjectivity of $\lambda $, there exists $y\in A$ such that $\lambda (Y)^{\top \top }=\lambda (y)^{\top }$. By Remark \ref{comuttop}, this is equivalent to $\lambda (Y^{\top \top })=\lambda (y^{\top })$, which, by Proposition \ref{topfiltru} and Lemma \ref{lambdaegal}, is equivalent to $Y^{\top \top }=y^{\top }$.

\noindent $(V)\Rightarrow (v^{\prime })$: Let $\emptyset \neq X\subseteq {\cal{L}}(A)$ such that $|X|\leq m$. The surjectivity of $\lambda $ implies that there exists $\emptyset \neq Y\subseteq A$ with $\lambda (Y)=X$ and $|Y|=|\lambda (Y)|=|X|\leq m$ ($Y$ can be chosen as in the proof of $(IV)\Rightarrow (iv^{\prime })$). Therefore $Y^{\top }\vee Y^{\top \top }=A$. By Remark \ref{comuttop}, Proposition \ref{`izomf`} and condition 4), this implies that $X^{\top }\vee X^{\top \top }=\lambda (Y)^{\top }\vee \lambda (Y)^{\top \top }=\lambda (Y^{\top })\vee \lambda (Y^{\top \top })=\lambda (Y^{\top }\vee Y^{\top \top })=\lambda (A)={\cal{L}}(A)$.

\noindent $(v^{\prime })\Rightarrow (V)$: Let $\emptyset \neq Y\subseteq A$ such that $|Y|\leq m$. Then $\lambda (Y)\neq \emptyset $ and $|\lambda (Y)|\leq |Y|\leq m$, so $\lambda (Y)^{\top }\vee \lambda (Y)^{\top \top }={\cal{L}}(A)$. By Remark \ref{comuttop}, Proposition \ref{`izomf`} and condition 4), this is equivalent to $\lambda (Y^{\top }\vee Y^{\top \top })=\lambda (A)$, which, by Lemma \ref{lambdaegal}, is equivalent to $Y^{\top }\vee Y^{\top \top }=A$.\end{proof}

A residuated lattice will be called an {\em $m$-co-Stone residuated lattice} iff the conditions of Theorem \ref{caractlrco-Stone} hold for it.\index{residuated lattice!$m$-co-Stone}

\begin{proposition}
$A$ is an $m$-co-Stone residuated lattice\index{residuated lattice!$m$-co-Stone} iff ${\cal{L}}(A)$ is an $m$-co-Stone lattice.\index{lattice!$m$-co-Stone}
\label{mco-Stoneiff}
\end{proposition}
\begin{proof}

This is a paraphrase of the equivalences established in the proof of Theorem \ref{caractlrco-Stone}.\end{proof}

The following two remarks show that co-Stone residuated lattices do not have a characterization like the one in \cite[Theorem 8.7.1, page 164]{bal} for co-Stone pseudocomplemented distributive lattices.

\begin{remark}
There exist co-Stone residuated lattices $A$ with elements $a\in A$ that do not satisfy the identity $\neg \, a\vee \neg \, \neg \, a=1$.\index{residuated lattice!co-Stone}
\label{co-Stonedar}

\end{remark}
\begin{proof}
Let us consider the residuated lattice $A$ from Example \ref{lrex0}. $B(A)=\{0,1\}$, $<0>=A$, $<1>=\{1\}$, $0^{\top }=a^{\top }=b^{\top }=c^{\top }=\{1\}$, $1^{\top }=A$, therefore $A$ is a co-Stone residuated lattice. But $\neg \, a=b$, $\neg \, \neg \, a=\neg \, b=a$, so $\neg \, a\vee \neg \, \neg \, a=b\vee a=c\neq 1$.\end{proof}

Notice that $A$ from the proof above is strongly co-Stone.

\begin{remark}
There exist residuated lattices $A$ that satisfy the identity $\neg \, a\vee \neg \, \neg \, a=1$ for all $a\in A$ and that are not co-Stone.\index{residuated lattice!co-Stone}

\end{remark}
\begin{proof}
Let $A$ be the residuated lattice in Example \ref{lrex8}. $A$ satisfies the identity in the enunciation. $B(A)=\{0,1\}$, $<0>=A$, $<1>=\{1\}$, but $c^{\top }=\{d,1\}$, hence $A$ is not co-Stone.\end{proof}

\begin{remark}
There exist residuated lattices $A$ that do not satisfy the identity $\neg \, a\vee \neg \, \neg \, a=1$ for all $a\in A$, but whose reticulations ${\cal{L}}(A)$ are pseudocomplemented lattices and satisfy this identity: $l^{*}\vee l^{**}=1$ for all $l\in {\cal{L}}(A)$.
\label{nucdefco-Stone}
\end{remark}

\begin{proof}
Let $A$ be the residuated lattice in Example \ref{lrex3}. For instance, $\neg \, b\vee \neg \, \neg \, b=c\vee b=c\neq 1$.

See its reticulation ${\cal{L}}(A)$ in Example \ref{ex3}. One can see that ${\cal{L}}(A)$ is pseudocomplemented and satisfies the identity in the enunciation, as it is a Boolean algebra.\end{proof}

The remark above shows that the alternate definition of co-Stone algebras, from \cite{rcig}, is not transferrable through the reticulation, which is the reason why we have chosen our definition over it.

\section{The Strongly Co-Stone Hull of a Residuated Lattice}
\label{hull}

\hspace*{11pt} This section is concerned with the construction of the strongly co-Stone hull of a residuated lattice and its properties, among which a universality property and its preservation by the reticulation functor.

In the following, let $A$ be a residuated lattice. We shall define the strongly co-Stone hull of $A$, in a manner similar to the definition from \cite{din3} for the strongly Stone hull of an MV-algebra.

Let $B={\rm CoAnn}(A)$. Let us consider the poset $\Pi (A)=P(B)$\index{$\Pi (A)$}\index{poset of finite partitions} of the finite partitions of $B$ (see Section \ref{preliminaries} for the definitions). For any ${\cal{C}}\in \Pi (A)$, set $\displaystyle A_{\cal{C}}=\prod _{C\in {\cal{C}}}A/(C^{\top })$. For every ${\cal{C}},{\cal{D}}\in \Pi (A)$ with ${\cal{C}}\leq {\cal{D}}$, we shall consider the map ${\cal{P}}_{\cal{CD}}:A_{\cal{C}}\rightarrow A_{\cal{D}}$, for all $(a_{C})_{C\in {\cal{C}}}\subseteq A$, ${\cal{P}}_{\cal{CD}}((a_{C}/(C^{\top }))_{C\in {\cal{C}}})=(b_{D}/(D^{\top }))_{D\in {\cal{D}}}$, where, by definition, for all $D\in {\cal{D}}$, $b_{D}=a_{C}$, where $C$ is the unique member of ${\cal{C}}$  such that $D\subseteq C$. It is immediate that ${\cal{P}}_{\cal{CD}}$ is an injective morphism of residuated lattices and that $((A_{\cal{C}})_{{\cal{C}}\in \Pi (A)},({\cal{P}}_{\cal{CD}})_{{\cal{C}}\leq {\cal{D}}})$ is an inductive system of residuated lattices. Let $\displaystyle \tilde{A}=\varinjlim _{{\cal{C}}\in \Pi (A)}A_{\cal{C}}$\index{$\tilde{A}$} be its inductive limit. By the uniqueness of the inductive limit, it follows that $\tilde{A}$ is unique up to a residuated lattice isomorphism.

\begin{definition}
We define $\tilde{A}$ to be the {\em strongly co-Stone hull of $A$}.\index{strongly co-Stone hull}\index{$\tilde{A}$}

\end{definition}

The reasons for adopting this name will be made apparent by the following results in this section. 

For every $a\in A$ and every ${\cal{C}}\in \Pi (A)$, we denote by $\epsilon (a)$\index{$\epsilon $}\index{$\epsilon (a)$} the congruence class $[a_{\cal{C}}]$\index{$[a_{\cal{C}}]$} in $\tilde{A}$ of the element $(a/(C^{\top }))_{C\in {\cal{C}}}$, element which we will denote $a_{\cal{C}}$.\index{$a_{\cal{C}}$} The definition of $\epsilon $ does not depend on ${\cal{C}}$, because, if ${\cal{D}}\in \Pi (A)$, then we have: $[a_{\cal{C}}]=[a_{\cal{D}}]$ iff there exists ${\cal{E}}\in \Pi (A)$ with ${\cal{C}},{\cal{D}}\leq {\cal{E}}$, such that $a_{\cal{E}}=a_{\cal{E}}$, which is true.

\begin{remark}
$\epsilon :A\rightarrow \tilde{A}$ is an injective morphism of residuated lattices.\index{$\epsilon $}
\end{remark}
\begin{proof}
Let $a,b\in A$, ${\cal{C}},{\cal{D}}\in \Pi (A)$ and ${\cal{E}}\in \Pi(A)$ such that ${\cal{C}},{\cal{D}}\leq {\cal{E}}$. Then: $\epsilon (a)\vee \epsilon(b)=[a_{\cal{C}}]\vee [b_{\cal{D}}]=[(a\vee b)_{\cal{E}}]=\epsilon(a\vee b)$. One can similarly prove that $\epsilon $ preserves the other residuated lattice operations, hence it is a morphism of residuated lattices. For the injectivity, we have that $\epsilon (a)=\epsilon (b)$ iff $[a_{\cal{C}}]=[b_{\cal{D}}]$ iff there exists ${\cal{F}}\in \Pi (A)$ with ${\cal{C}},{\cal{D}}\leq {\cal{F}}$ such that $a_{\cal{F}}=b_{\cal{F}}$ iff, for all $F\in {\cal{F}}$, $a/(F^{\top })=b/(F^{\top })$ iff, for all $F\in {\cal{F}}$, $a\leftrightarrow b\in F^{\top }$ iff $\displaystyle a\leftrightarrow b\in \bigcap _{F\in {\cal{F}}}F^{\top }=\left(\bigcup _{F\in {\cal{F}}}F \right) ^{\top }=A^{\top }=\{1\}$ iff $a=b$, by Lemma \ref{calcul}, (\ref{art5(iv)}).\end{proof}

One can identify $A$ and $\epsilon (A)$.

\begin{remark}
For all $C\in {\rm CoAnn}(A)$, $\{C,C^{\top }\}$ is a partition of $B={\rm CoAnn}(A)$ and $C^{\top \top }=C$.
\end{remark}
\begin{proof}
Obviously $C\cap C^{\top }=\{1\}$, and $C\vee ^{\top }C^{\top }=(C^{\top }\cap C^{\top \top })^{\top }=1^{\top }=A$. 

Obviously $C\subseteq C^{\top \top }$, but $C=D^{\top }$ for some $D\subseteq A$, so $C^{\top }=D^{\top \top }\supseteq D$, which implies $C^{\top \top }\subseteq D^{\top }=C$, hence $C^{\top \top }=C$.\end{proof}

\begin{lemma}
For all $x\in \tilde{A}$, there exists $n\in \N ^{*}$ and there exist $e_{1},\ldots ,e_{n}\in B(\tilde{A})$ and $a_{1},\ldots ,a_{n}\in A$ such that: for all $i,j\in \overline{1,n}$ with $i\neq j$, $e_{i}\vee e_{j}=1$, $\displaystyle \bigwedge _{i=1}^{n}e_{i}=0$ and $\displaystyle x=\bigwedge _{i=1}^{n}(\epsilon (a_{i})\vee e_{i})$.
\label{exprx}

\end{lemma}

\begin{proof}
Let $x\in \tilde{A}$. Then there exists ${\cal{C}}=\{C_{1},\ldots ,C_{n}\}\in \Pi(A)$ and there exist $a_{1},\ldots ,a_{n}\in A$ such that $x=[(a_{1}/C_{1}^{\top },\ldots ,a_{n}/C_{n}^{\top })]$.

For every $i\in \overline{1,n}$, let ${\cal{D}}_{i}=\{C_{i},C_{i}^{\top }\}\in \Pi (A)$, so $A_{{\cal{D}}_{i}}=A/C_{i}^{\top }\times A/C_{i}$, Obviously, for each $i$, ${\cal{D}}_{i}\leq {\cal{C}}$ (the distributivity of the Boolean algebra ${\rm CoAnn}(A)$ and the definition of the complement imply that $\vee ^{\top }_{j\neq i}C_{j}=C_{i}^{\top }$) and ${\cal{P}}_{{\cal{D}}_{i},{\cal{C}}}: A_{{\cal{D}}_{i}}\rightarrow A_{\cal{C}}$, for all $a,b\in A$, ${\cal{P}}_{{\cal{D}}_{i},{\cal{C}}}(a/C_{i}^{\top },b/C_{i})=(b/C_{1}^{\top },\ldots ,$\linebreak $b/C_{i-1}^{\top },a/C_{i}^{\top },b/C_{i+1}^{\top },\ldots ,b/C_{n}^{\top })$. For every $i\in \overline{1,n}$, let $d_{i}={\cal{P}}_{{\cal{D}}_{i},{\cal{C}}}(0/C_{i}^{\top },$\linebreak $1/C_{i})=(1/C_{1}^{\top },\ldots ,1/C_{i-1}^{\top },0/C_{i}^{\top },1/C_{i+1}^{\top },\ldots ,1/C_{n}^{\top })$ and let $e_{i}=[(0/C_{i}^{\top },$\linebreak $1/C_{i})]$. Notice that: for all $i,j\in \overline{1,n}$ with $i\neq j$, $e_{i}\vee e_{j}=1$ and $\displaystyle \bigwedge _{i=1}^{n}e_{i}=0$. For all $i\in \overline{1,n}$, $\epsilon (a_{i})\vee e_{i}=[(a_{i}/C_{i}^{\top },a_{i}/C_{i})]\vee [(0/C_{i}^{\top },1/C_{i})]=[(a_{i}/C_{i}^{\top },1/C_{i})]$ and ${\cal{P}}_{{\cal{D}}_{i},{\cal{C}}}(a_{i}/C_{i}^{\top },1/C_{i})=(1/C_{1}^{\top },\ldots ,1/C_{i-1}^{\top },a_{i}/C_{i}^{\top },$\linebreak $1/C_{i+1}^{\top },\ldots ,1/C_{n}^{\top })$, so $\displaystyle \bigwedge _{i=1}^{n}(\epsilon (a_{i})\vee e_{i})=\bigwedge _{i=1}^{n}[(a_{i}/C_{i}^{\top },1/C_{i})]=\bigwedge _{i=1}^{n}[{\cal{P}}_{{\cal{D}}_{i},{\cal{C}}}(a_{i}/C_{i}^{\top },$\linebreak $\displaystyle 1/C_{i})]=[\bigwedge _{i=1}^{n}{\cal{P}}_{{\cal{D}}_{i},{\cal{C}}}(a_{i}/C_{i}^{\top },1/C_{i})]=[(a_{1}/C_{1}^{\top },\ldots ,a_{n}/C_{n}^{\top })]=x$.\end{proof}

\begin{lemma}
$\tilde{A}$ is a co-Stone residuated lattice.
\label{esteco-Stone}
\end{lemma}
\begin{proof}
For all $a\in A$, we denote $\epsilon (a)^{*}=[(0/a^{\top \top },1/a^{\top })]$.\index{$\epsilon (a)^{*}$} Let $a\in A$ and let us prove that $\epsilon (a)^{\top }=<\epsilon (a)^{*}>$. Let $x\in \tilde{A}$, arbitrary, so, by Lemma \ref{exprx}, there exists $n\in \N ^{*}$ and there exist $e_{1},\ldots ,e_{n}\in B(\tilde{A})$ and $a_{1},\ldots ,a_{n}\in A$, chosen as in the proof of Lemma \ref{exprx}, such that: for all $i,j\in \overline{1,n}$ with $i\neq j$, $e_{i}\vee e_{j}=1$, $\displaystyle \bigwedge _{i=1}^{n}e_{i}=0$ and $\displaystyle x=\bigwedge _{i=1}^{n}(\epsilon (a_{i})\vee e_{i})$.

Let $i\in \overline{1,n}$. Notice that $\{a^{\top }\cap C_{i}^{\top },a^{\top \top }\cap C_{i}^{\top },a^{\top }\cap C_{i},a^{\top \top }\cap C_{i}\}$ is a partition and it is a refinement of each of the partitions $\{a^{\top },a^{\top \top}\}$ and $\{C_{i}^{\top },C_{i}\}$. $$\epsilon (a_{i})\vee e_{i}=[(a_{i}/C_{i}^{\top },a_{i}/C_{i})]\vee [(0/C_{i}^{\top },1/C_{i})]=[(a_{i}/C_{i}^{\top },1/C_{i})]$$  $$=[(1/(a^{\top }\cap C_{i}^{\top })^{\top },1/(a^{\top \top }\cap C_{i}^{\top })^{\top },a_{i}/(a^{\top }\cap C_{i})^{\top },a_{i}/(a^{\top \top }\cap C_{i})^{\top })],$$ so $$\epsilon (a_{i})\vee e_{i}\vee \epsilon (a)^{*}=$$ $$[(1/(a^{\top }\cap C_{i}^{\top })^{\top },1/(a^{\top \top }\cap C_{i}^{\top })^{\top },a_{i}/(a^{\top }\cap C_{i})^{\top },a_{i}/(a^{\top \top }\cap C_{i})^{\top })]\vee $$ $$[(0/(a^{\top }\cap C_{i}^{\top })^{\top },1/(a^{\top \top }\cap C_{i}^{\top })^{\top },0/(a^{\top }\cap C_{i})^{\top },1/(a^{\top \top }\cap C_{i})^{\top })]=$$ $$[(1/(a^{\top }\cap C_{i}^{\top })^{\top },1/(a^{\top \top }\cap C_{i}^{\top })^{\top },a_{i}/(a^{\top }\cap C_{i})^{\top },1/(a^{\top \top }\cap C_{i})^{\top })].$$ $$\epsilon (a_{i})\vee e_{i}\vee \epsilon (a)=[(a_{i}/C_{i}^{\top },1/C_{i})]\vee [(a/C_{i}^{\top },a/C_{i})]=[((a_{i}\vee a)/C_{i}^{\top },1/C_{i})].$$ For the following, see Lemma \ref{principal} and Lemma \ref{prop-cong}, (\ref{aF=0Fsau1F}). Hence: $\epsilon (a_{i})\vee e_{i}\in \epsilon (a)^{\top}$ iff $\epsilon (a_{i})\vee e_{i}\vee \epsilon(a)=1$ iff $(a_{i}\vee a)/C_{i}^{\top }=1/C_{i}^{\top }$ iff $a_{i}\vee a\in C_{i}^{\top }$, and, since $\epsilon (a)^{*}$ is obviously idempotent, $\epsilon (a_{i})\vee e_{i}\in <\epsilon (a)^{*}>$ iff $\epsilon (a)^{*}\leq \epsilon (a_{i})\vee e_{i}$ iff $\epsilon (a_{i})\vee e_{i}\vee \epsilon (a)^{*}=\epsilon (a_{i})\vee e_{i}$ iff $a_{i}/(a^{\top \top }\cap C_{i})^{\top }=1/(a^{\top \top }\cap C_{i})^{\top }$ iff $a_{i}\in (a^{\top \top }\cap C_{i})^{\top }$. We aim to prove that $\epsilon (a_{i})\vee e_{i}\in \epsilon (a)^{\top}$ iff $\epsilon (a_{i})\vee e_{i}\in <\epsilon (a)^{*}>$; it is sufficient to show that $a_{i}\vee a\in C_{i}^{\top }$ iff $a_{i}\in (a^{\top \top }\cap C_{i})^{\top }$. Let us prove the direct implication; so let us assume that $a_{i}\vee a\in C_{i}^{\top }$. Let $t\in a^{\top \top}\cap C_{i}$, arbitrary. Since $t\in a^{\top \top }$, it follows that $t\vee a_{i}\in a^{\top \top }$. Since $t\in C_{i}$, we get $t\vee a_{i}\vee a=1$, hence $t\vee a_{i}\in a^{\top }$. So $t\vee a_{i}\in a^{\top \top }\cap a^{\top }=\{1\}$, hence $t\vee a_{i}=1$. Therefore $a_{i}\in (a^{\top \top }\cap C_{i})^{\top }$. For the converse implication, let us assume that $a_{i}\in (a^{\top \top }\cap C_{i})^{\top }$. Let $t\in C_{i}$, arbitrary, so $t\vee a\in C_{i}$. $a\in a^{\top \top }$, so $t\vee a\in a^{\top \top }$. Therefore $t\vee a\in a^{\top \top }\cap C_{i}$. Then $t\vee a\vee a_{i}=1$, so $a\vee a_{i}\in C_{i}^{\top }$. Thus $\epsilon (a_{i})\vee e_{i}\in \epsilon (a)^{\top }$ iff $\epsilon (a_{i})\vee e_{i}\in <\epsilon (a)^{*}>$. 

Now we can go one step further and prove that $x\in \epsilon (a)^{\top }$ iff $x\in <\epsilon (a)^{*}>$, which leads to the conclusion that $\epsilon (a)^{\top }=<\epsilon (a)^{*}>$, as $x$ is arbitrary in $\tilde{A}$. By Remark \ref{inF} and the above, $x\in \epsilon (a)^{\top }$ iff $\displaystyle \bigwedge _{i=1}^{n}(\epsilon (a_{i})\vee e_{i})\in \epsilon (a)^{\top }$ iff, for all $i\in \overline{1,n}$, $\epsilon (a_{i})\vee e_{i}\in \epsilon (a)^{\top }$ iff, for all $i\in \overline{1,n}$, $\epsilon (a_{i})\vee e_{i}\in <\epsilon (a)^{*}>$ iff $\displaystyle \bigwedge _{i=1}^{n}(\epsilon (a_{i})\vee e_{i})\in <\epsilon (a)^{*}>$ iff $x\in <\epsilon (a)^{*}>$. Hence $\epsilon (a)^{\top }=<\epsilon (a)^{*}>$.

Let $y\in \tilde{A}$, arbitrary. We shall prove that $y^{\top }$ is generated by an element from the Boolean center of $\tilde{A}$, which will end the proof. First, let us write $y$ as an expresion made of elements of $A$ and elements from the Boolean center of $\tilde{A}$. By Lemma \ref{exprx}, there exists $m\in \N ^{*}$ and there exist $f_{1},\ldots ,f_{m}\in B(\tilde{A})$ and $b_{1},\ldots ,b_{m}\in A$, chosen like in the proof of the lemma, such that: for all $i,j\in \overline{1,m}$ with $i\neq j$, $f_{i}\vee f_{j}=1$, $\displaystyle \bigwedge _{i=1}^{m}f_{i}=0$ and $\displaystyle y=\bigwedge _{i=1}^{m}(\epsilon (b_{i})\vee f_{i})$. Next we shall obtain a writing of $y^{\top }$ depending on the $b_{i}$s and the $f_{i}$s. Let $v\in \tilde{A}$, arbitrary. By Remark \ref{inF}, $v\in y^{\top }$ iff $v\vee y=1$ iff $y\in v^{\top }$ iff $\displaystyle \bigwedge _{i=1}^{m}(\epsilon (b_{i})\vee f_{i})\in v^{\top }$ iff, for all $i\in \overline{1,m}$, $\epsilon (b_{i})\vee f_{i}\in v^{\top }$ iff, for all $i\in \overline{1,m}$, $\epsilon (b_{i})\vee f_{i}\vee v=1$ iff $\displaystyle v\in \bigcap _{i=1}^{m}(\epsilon (b_{i})\vee f_{i})^{\top }$. Therefore $\displaystyle y^{\top }=\bigcap _{i=1}^{m}(\epsilon (b_{i})\vee f_{i})^{\top }$.

Now we shall write each co-annihilator filter from this writing of $y^{\top }$ as a principal filter generated by an element from the Boolean center of $\tilde{A}$.

For every $i\in \overline{1,m}$, let us denote $d_{i}=\epsilon (b_{i})^{*}$, which is idempotent; then, by the above, $\epsilon (b_{i})^{\top }=<d_{i}>$, with $d_{i}$ idempotent, for every $i$. Moreover, each $d_{i}\in B(\tilde{A})$, as Lemma \ref{bcomut}, (\ref{B(A)=}) shows.

Let $z\in \tilde{A}$, arbitrary. By Lemma \ref{bcomut}, (\ref{evx}), Lemma \ref{principal} and the law of residuation, for each $i\in \overline{1,m}$, $z\in (\epsilon (b_{i})\vee f_{i})^{\top }$ iff $z\vee \epsilon (b_{i})\vee f_{i}=1$ iff $z\vee f_{i}\in \epsilon (b_{i})^{\top }=<d_{i}>$ iff $d_{i}\leq z\vee f_{i}=\neg \, f_{i}\rightarrow z$ iff $d_{i}\odot \neg \, f_{i}\leq z$ iff $z\in <d_{i}\odot \neg \, f_{i}>$ (remember that $\neg \, f_{i}$ is also idempotent). Hence, for each $i\in \overline{1,m}$, $(\epsilon (b_{i})\vee f_{i})^{\top }=<d_{i}\odot \neg \, f_{i}>$, with $d_{i}\odot \neg \, f_{i}\in B(\tilde{A})$, as Lemma \ref{bcomut}, (\ref{BA-e-f}) shows.

Therefore, by Proposition \ref{`sisau`}, $\displaystyle y^{\top }=\bigcap _{i=1}^{m}(\epsilon (b_{i})\vee f_{i})^{\top }=\bigcap _{i=1}^{m}<d_{i}\odot \neg \, f_{i}>=<\bigvee _{i=1}^{m}(d_{i}\odot \neg \, f_{i})>$, with $\displaystyle \bigvee _{i=1}^{m}(d_{i}\odot \neg \, f_{i})\in B(\tilde{A})$.

Hence $\tilde{A}$ is a co-Stone residuated lattice.\end{proof}

\begin{lemma}
Let $C,D\in {\rm CoAnn}(A)$. Then $[(0/C^{\top },1/C)]\vee [(0/D^{\top },1/D)]=[(0/(C\cap D)^{\top },1/(C\cap D))]$.
\label{cveed}
\end{lemma}
\begin{proof}
We will use the following refinement of the partitions $\{C^{\top },C\}$ and $\{D^{\top },D\}$: $\{C^{\top }\cap D^{\top },C^{\top }\cap D,C\cap D^{\top },C\cap D\}$. $$[(0/C^{\top },1/C)]\vee [(0/D^{\top },1/D)]=$$ $$[(1/(C^{\top }\cap D^{\top })^{\top },1/(C^{\top }\cap D)^{\top },0/(C\cap D^{\top })^{\top },0/(C\cap D)^{\top })]\vee $$ $$[(1/(C^{\top }\cap D^{\top })^{\top },0/(C^{\top }\cap D)^{\top },1/(C\cap D^{\top })^{\top },0/(C\cap D)^{\top })]=$$ $$[(1/(C^{\top }\cap D^{\top })^{\top },1/(C^{\top }\cap D)^{\top },1/(C\cap D^{\top })^{\top },0/(C\cap D)^{\top })]=$$ $$[(0/(C\cap D) ^{\top },1/(C\cap D))].$$\end{proof}

Here is a generalization of the previous lemma (see Proposition \ref{`sisau`}).

\begin{lemma}
Let $I$ be an arbitrary index set and, for all $i\in I$, $E_{i}\in {\rm CoAnn}(A)$, $e_{i}=[(0/E_{i}^{\top },1/E_{i})]$, and $\displaystyle e=[(0/(\bigcap _{i\in I}E_{i})^{\top },1/(\bigcap _{i\in I}E_{i}))]$. Then: $\displaystyle \bigcap _{i\in I}<e_{i}>=<e>$.
\label{intersET}
\end{lemma}
\begin{proof}
Let $x\in \tilde{A}$, arbitrary. We shall prove that: $\displaystyle x\in \bigcap _{i\in I}<e_{i}>$ iff $x\in <e>$. By the proof of Lemma \ref{exprx}, there exists an $n\in \N ^{*}$ and, for all $j\in \overline{1,n}$, there exists $a_{j}\in A$ and $C_{j}\in {\rm CoAnn}(A)$, such that $\displaystyle x=\bigwedge _{j=1}^{n}[(a_{j}/C_{j}^{\top },1/C_{j})]$. By Remark \ref{inF}, it is sufficient to show that, for each $j\in \overline{1,n}$, $\displaystyle [(a_{j}/C_{j}^{\top },1/C_{j})]\in \bigcap _{i\in I}<e_{i}>$ iff $[(a_{j}/C_{j}^{\top },1/C_{j})]\in <e>$. Let $j\in \overline{1,n}$. Obviously, $e$ and each $e_{i}$ are idempotent (actually, by Lemma \ref{bcomut}, (\ref{B(A)=}), they belong to $B(\tilde{A})$), therefore it is sufficient to show that $e\leq [(a_{j}/C_{j}^{\top },1/C_{j})]$ iff, for all $i\in I$, $e_{i}\leq [(a_{j}/C_{j}^{\top },1/C_{j})]$ (see Lemma \ref{principal}).

Let $i\in I$. $e_{i}\leq [(a_{j}/C_{j}^{\top },1/C_{j})]$ iff $[(0/E_{i}^{\top },1/E_{i})]\leq [(a_{j}/C_{j}^{\top },1/C_{j})]$ iff $[(0/(E_{i}\cap C_{j})^{\top },1/(E_{i}^{\top }\cap C_{j})^{\top },0/(E_{i}\cap C_{j}^{\top })^{\top },1/(E_{i}^{\top }\cap C_{j}^{\top })^{\top })]\leq [(a_{j}/(E_{i}\cap C_{j})^{\top },a_{j}/(E_{i}^{\top }\cap C_{j})^{\top },1/(E_{i}\cap C_{j}^{\top })^{\top },1/(E_{i}^{\top }\cap C_{j}^{\top })^{\top })]$ iff $1/(E_{i}^{\top }\cap C_{j})^{\top }\leq a_{j}/(E_{i}^{\top }\cap C_{j})^{\top }$ iff $1/(E_{i}^{\top }\cap C_{j})^{\top }=a_{j}/(E_{i}^{\top }\cap C_{j})^{\top }$ iff $a_{j}\in (E_{i}^{\top }\cap C_{j})^{\top }$ (see Lemma \ref{prop-cong}, (\ref{aF=0Fsau1F})). Analogously, $e\leq [(a_{j}/C_{j}^{\top },1/C_{j})]$ iff $\displaystyle a_{j}\in ((\bigcap _{i\in I}E_{i})^{\top }\cap C_{j})^{\top }$. So it remains to show that $\displaystyle a_{j}\in ((\bigcap _{i\in I}E_{i})^{\top }\cap C_{j})^{\top }$ iff, for all $i\in I$, $a_{j}\in (E_{i}^{\top }\cap C_{j})^{\top }$, which is true, because, by the definition of the co-annihilator of a set: $\displaystyle (\bigcap _{i\in I}E_{i})^{\top }=\bigcup _{i\in I}E_{i}^{\top }$, hence $\displaystyle (\bigcap _{i\in I}E_{i})^{\top }\cap C_{j}=\bigcup _{i\in I}E_{i}^{\top }\cap C_{j}$, that is $\displaystyle (\bigcap _{i\in I}E_{i})^{\top }\cap C_{j}=\bigcup _{i\in I}(E_{i}^{\top }\cap C_{j})$, thus $\displaystyle ((\bigcap _{i\in I}E_{i})^{\top }\cap C_{j})^{\top }=(\bigcup _{i\in I}(E_{i}^{\top }\cap C_{j}))^{\top }=\bigcap _{i\in I}(E_{i}^{\top }\cap C_{j})^{\top }$.\end{proof}

\begin{lemma}
$\tilde{A}$ is a strongly co-Stone residuated lattice.
\label{estetareco-Stone}

\end{lemma}
\begin{proof}
Let $X\subseteq \tilde{A}$. By Lemma \ref{esteco-Stone}, $\displaystyle X^{\top }=\bigcap _{x\in X}x^{\top }=\bigcap _{x\in X}<e_{x}>$ for some $e_{x}\in B(\tilde{A})$ for every $x\in X$.

Let $x\in X$. Using the notations from the proofs of Lemmas \ref{exprx} and \ref{esteco-Stone}, $e_{x}$ is of the form $\displaystyle e_{x}=\bigvee _{i=1}^{m}(d_{i}\odot \neg \, f_{i})$, where $d_{i}=\epsilon (b_{i})^{*}=[(0/b_{i}^{\top \top },1/b_{i}^{\top })]$, with $b_{i}\in A$, and $f_{i}=[(0/C_{i}^{\top },1/C_{i})]$, with $C_{i}\in {\rm CoAnn}(A)$. So $\neg \, f_{i}=[(1/C_{i}^{\top },0/C_{i})]$ and $$d_{i}\odot \neg \, f_{i}=[(1/(b_{i}^{\top \top }\cap C_{i}^{\top })^{\top },1/(b_{i}^{\top \top }\cap C_{i})^{\top },0/(b_{i}^{\top }\cap C_{i}^{\top })^{\top },0/(b_{i}^{\top }\cap C_{i})^{\top })]$$ $$\odot [(0/(b_{i}^{\top \top }\cap C_{i}^{\top })^{\top },1/(b_{i}^{\top \top }\cap C_{i})^{\top },0/(b_{i}^{\top }\cap C_{i}^{\top })^{\top },1/(b_{i}^{\top }\cap C_{i})^{\top })]=$$ $$[(0/(b_{i}^{\top \top }\cap C_{i}^{\top })^{\top },1/(b_{i}^{\top \top }\cap C_{i})^{\top },0/(b_{i}^{\top }\cap C_{i}^{\top })^{\top },0/(b_{i}^{\top }\cap C_{i})^{\top })]=$$ $$[(0/(b_{i}^{\top \top }\cap C_{i}),1/(b_{i}^{\top \top }\cap C_{i})^{\top })]=[(0/D_{i}^{\top },1/D_{i})],$$ where $D_{i}=(b_{i}^{\top \top }\cap C_{i})^{\top }$. By Lemma \ref{cveed}, for every $i,j\in \overline{1,m}$ with $i\neq j$, $[(0/D_{i}^{\top },1/D_{i})]\vee [(0/D_{j}^{\top },1/D_{j})]=[(0/(D_{i}\cap D_{j})^{\top },1/(D_{i}\cap D_{j}))]$, hence, by induction on $m$, one can show that $$e_{x}=\bigvee _{i=1}^{m}(d_{i}\odot \neg \, f_{i})=[(0/(\bigcap _{i=1}^{m}D_{i})^{\top },1/(\bigcap _{i=1}^{m}D_{i}))]=[(0/E_{x}^{\top },1/E_{x})],$$ where $\displaystyle E_{x}=\bigcap _{i=1}^{m}D_{i}$.

Let $\displaystyle e=[(0/(\bigcap _{x\in X}E_{x})^{\top },1/(\bigcap _{x\in X}E_{x}))]$. Lemma \ref{bcomut}, (\ref{B(A)=}) ensures us that $e\in B(\tilde{A})$. By Lemma \ref{intersET}, $\displaystyle <e>=\bigcap _{x\in X}<e_{x}>=X^{\top }$. So $\tilde{A}$ is strongly co-Stone.\end{proof}

\begin{proposition}
$A$ is co-dense in $\tilde{A}$.

\end{proposition}
\begin{proof}
Let $x=[a]\in \tilde{A}\setminus \{1\}$, with $a=(a_{C}/C^{\top })
_{C\in {\cal{C}}}\in A_{\cal{C}}$, ${\cal{C}}\in \Pi (A)$. What we have to do is prove that there exists $y\in A$ such that $x\leq \epsilon (y)<1$.

$x\neq 1$, so there exists $C\in {\cal{C}}$ such that $a_{C}/C^{\top }\neq 1/C^{\top }$, that is $a_{C}\notin C^{\top }$ (see Lemma \ref{prop-cong}, (\ref{aF=0Fsau1F})), which means that there exists an element $b\in C$ such that $b\vee a_{C}\neq 1$. Set $y=b\vee a_{C}\in C$, since $b\in C$ and $C\in {\rm CoAnn}(A)\subseteq {\cal{F}}(A)$, by Proposition \ref{topfiltru}. Let $D\in {\cal{C}}\setminus \{C\}$, arbitrary, so $C\cap D=\{1\}$, hence $C\subseteq D^{\top }$ (see Proposition \ref{CoAnnbool}; the complement $D^{\top }$ of $D$ in the Boolean algebra ${\rm CoAnn}(A)$ equals its pseudocomplement). $y\in C$, therefore $y\in D^{\top }$, hence $y/D^{\top }=1/D^{\top }$, by Lemma \ref{prop-cong}, (\ref{aF=0Fsau1F}). But $C\cap C^{\top }=\{1\}$ and $y\neq 1$, so $y\in C\setminus \{1\}$, so $y\notin C^{\top }$, that is $y/C^{\top }\neq 1/C^{\top }$, by Lemma \ref{prop-cong}, (\ref{aF=0Fsau1F}).

So we have proven that $(y/D^{\top })_{D\in {\cal{C}}}$ has exactly one component different from 1, namely $y/C^{\top }$, and this component equals $(b\vee a_{C})/C^{\top }$, so it is greater than $a_{C}/C^{\top }$. Therefore we have: $x=[(a_{D}/D^{\top })_{D\in {\cal{C}}}]\leq [(y/D^{\top })_{D\in {\cal{C}}}]=\epsilon (y)<1$.\end{proof}

\begin{conjecture}\index{strongly co-Stone hull}
The strongly co-Stone hull of $A$, $\tilde{A}$, verifies the following universality property: for any strongly co-Stone residuated lattice $A_{1}$ and any morphism of residuated lattices $f:A\rightarrow A_{1}$ with the property that, for any $X\subseteq A$, we have $f(X^{\top })=f(X)^{\top }$, there exists a unique morphism of residuated lattices $\overline{f}:\tilde{A}\rightarrow A_{1}$ such that $\overline{f}\circ \epsilon =f$.

\begin{center}
\begin{picture}(60,60)(0,0)
\put(7,37){$A$}
\put(20,40){\vector(1,0){20}}
\put(42,37){$\tilde{A}$}
\put(28,44){$\epsilon $}
\put(45,35){\vector(0,-1){20}}

\put(47,22){$\overline{f}$}
\put(42,5){$A_{1}$}
\put(18,35){\vector(1,-1){22}}
\put(16,22){$f$}
\end{picture}
\end{center}
\label{univco-Stone}

\end{conjecture}

A similar definition can be given for the unique strongly co-Stone hull of a bounded distributive lattice.\index{strongly co-Stone hull} This definition is in accordance with the one from \cite{dav}.

\begin{proposition}
${\cal{L}}$ preserves the strongly co-Stone hull, namely the reticulation of the strongly co-Stone hull of a residuated lattice equals the strongly co-Stone hull of the reticulation of that residuated lattice.\index{strongly co-Stone hull} 
\end{proposition}

\begin{proof}
Let ${\cal{L}}(\tilde{A})$ be the reticulation of $\tilde{A}$ and $\widetilde{{\cal{L}}(A)}$ be the strongly co-Stone hull of the reticulation $({\cal{L}}(A),\lambda )$ of $A$. We will prove that there exists a bounded lattice isomorphism from $\widetilde{{\cal{L}}(A)}$ to ${\cal{L}}(\tilde{A})$.

The order isomorphism $\mu $ from the proof of Proposition \ref{izommu} obviously induces an order isomorphism $\nu :\Pi (A)\rightarrow \Pi({\cal{L}}(A))$, for all ${\cal{C}}\in \Pi (A)$, $\nu ({\cal{C}})=\{\mu (C)|C\in {\cal{C}}\}$. $\displaystyle \tilde{A}=\varinjlim _{{\cal{C}}\in \Pi (A)}(A_{\cal{C}},{\cal{P}}_{\cal{CD}})$, hence, by Proposition \ref{limind}, $\displaystyle {\cal{L}}(\tilde{A})=\varinjlim _{{\cal{C}}\in \Pi (A)}({\cal{L}}(A_{\cal{C}}),{\cal{L}}({\cal{P}}_{\cal{CD}}))$, and $\displaystyle \widetilde{{\cal{L}}(A)}=\varinjlim _{\stackrel{\scriptstyle {\cal{E}},{\cal{F}}\in \Pi ({\cal{L}}(A)),}{\scriptstyle {\cal{E}}\leq {\cal{F}}}}({\cal{L}}(A)_{\cal{E}},{\cal{Q}}_{\cal{EF}})=\varinjlim _{\stackrel{\scriptstyle {\cal{C}},{\cal{D}}\in \Pi (A),}{\scriptstyle {\cal{C}}\leq {\cal{D}}}}({\cal{L}}(A)_{\nu ({\cal{C}})},{\cal{Q}}_{\nu ({\cal{C}})\nu ({\cal{D}})})$, where, in conformity to the construction of the strongly co-Stone hull of a bounded distributive lattice, the ${\cal{Q}}_{\cal{EF}}$ are defined this way: for all ${\cal{E}},{\cal{F}}\in \Pi ({\cal{L}}(A))$ with ${\cal{E}}\leq {\cal{F}}$, for all $(a_{E})_{E\in {\cal{E}}}\subseteq A$, ${\cal{Q}}_{\cal{EF}}((\lambda (a_{E})/E^{\top })_{E\in {\cal{E}}})=(\lambda (a_{F})/F^{\top })_{F\in {\cal{F}}}$, where $a_{F}=a_{E}$ iff $F\subseteq E$. It follows that, for all ${\cal{C}},{\cal{D}}\in \Pi (A)$ with ${\cal{C}}\leq {\cal{D}}$, ${\cal{Q}}_{\nu ({\cal{C}})\nu ({\cal{D}})}$ is defined as follows: for all $x=(\lambda (a_{\mu (C)})/\mu (C)^{\top })_{C\in {\cal{C}}}\in {\cal{L}}(A)_{\nu ({\cal{C}})}$, with $(a_{\mu (C)})_{C\in {\cal{C}}}\subseteq A$, ${\cal{Q}}_{\nu ({\cal{C}})\nu ({\cal{D}})}(x)=(\lambda (a_{\mu (D)})/\mu (D)^{\top })_{D\in {\cal{D}}}=(\lambda (a_{\mu (D)})/\lambda (D)^{\top })_{D\in {\cal{D}}}=$\linebreak $(\lambda (a_{\mu (D)})/\lambda (D^{\top }))_{D\in {\cal{D}}}$, where $a_{\mu (D)}=a_{\mu (C)}$ iff $\mu (D)\subseteq \mu (C)$ iff $D\subseteq C$ (since $\mu $ is an order isomorphism); we have used Remark \ref{comuttop}.

Let ${\cal{C}}\in \Pi (A)$, arbitrary but fixed. $\displaystyle A_{\cal{C}}=\prod _{C\in {\cal{C}}}A/C^{\top }$ and $\displaystyle {\cal{L}}(A)_{\nu ({\cal{C}})}=\prod _{C\in {\cal{C}}}{\cal{L}}(A)/\mu (C)^{\top }=\prod _{C\in {\cal{C}}}{\cal{L}}(A)/\lambda (C)^{\top }=\prod _{C\in {\cal{C}}}{\cal{L}}(A)/\lambda (C^{\top })$, as Remark \ref{comuttop} shows. By Proposition \ref{findirprod}, $\displaystyle {\cal{L}}(A_{\cal{C}})=\prod _{C\in {\cal{C}}}{\cal{L}}(A/C^{\top })$. By Proposition \ref{presquot}, for each $C\in {\cal{C}}$, the function $h_{C}:{\cal{L}}(A)/\lambda (C^{\top })\rightarrow {\cal{L}}(A/C^{\top })$, for all $a\in A$, $h_{C}(\lambda (a)/\lambda (C^{\top }))=\lambda _{C}(a/C^{\top })$ is a bounded lattice isomorphism, where we denoted by $({\cal{L}}(A/C^{\top }),\lambda _{C})$ the reticulation of $A/C^{\top }$. Let $f_{\cal{C}}:{\cal{L}}(A)_{\nu ({\cal{C}})}\rightarrow {\cal{L}}(A_{\cal{C}})$, for all $x$ as above, $f_{\cal{C}}(x)=(h_{C}(\lambda(a_{\mu (C)})/\lambda (C^{\top })))_{C\in {\cal{C}}}=$\linebreak $(\lambda _{C}(a_{\mu (C)}/C^{\top }))_{C\in {\cal{C}}}$. The fact that each $h_{C}$ is well defined and it is a bounded lattice isomorphism implies that $f_{\cal{C}}$ is well defined and it is a bounded lattice isomorphism.

Let ${\cal{C}},{\cal{D}}\in \Pi (A)$ with ${\cal{C}}\leq {\cal{D}}$, arbitrary but fixed. Let $x$ be as above and $(a_{\mu (D)})_{D\in {\cal{D}}}\subseteq A$, with $a_{\mu (D)}=a_{\mu (C)}$ iff $\mu (D)\subseteq \mu (C)$ iff $D\subseteq C$ (since $\mu $ is an order isomorphism).

Then $f_{\cal{D}}({\cal{Q}}_{\nu ({\cal{C}})\nu ({\cal{D}})}(x))=f_{\cal{D}}((\lambda (a_{\mu (D)})/\lambda (D^{\top }))_{D\in {\cal{D}}})=(\lambda _{D}(a_{\mu (D)}/$\linebreak $D^{\top }))_{D\in {\cal{D}}}$.

On the other hand, according to Proposition \ref{findirprod}, $\displaystyle ({\cal{L}}(A_{\cal{C}}),\prod _{C\in {\cal{C}}}\lambda _{C})$ is the reticulation of $A_{\cal{C}}$ and $\displaystyle ({\cal{L}}(A_{\cal{D}}),\prod _{D\in {\cal{D}}}\lambda _{D})$ is the reticulation of $A_{\cal{D}}$, thus, by Proposition \ref{`lmorfisme`} and Definition \ref{deflmorfisme}, we obtain the commutative diagram below and thus the equalities that follow it:

\begin{center}
\begin{picture}(170,73)(0,0)
\put(0,34){$\displaystyle \prod _{C\in {\cal{C}}}\lambda _{C}$}
\put(35,45){\vector(0,-1){24}}
\put(30,10){${\cal{L}}(A_{\cal{C}})$}
\put(30,48){$A_{\cal{C}}$}
\put(60,15){\vector(1,0){48}}
\put(65,5){${\cal{L}}({\cal{P}}_{\cal{CD}})$}
\put(44,53){\vector(1,0){64}}
\put(65,57){${\cal{P}}_{\cal{CD}}$}
\put(110,10){${\cal{L}}(A_{\cal{D}})$}
\put(110,48){$A_{\cal{D}}$}
\put(115,45){\vector(0,-1){24}}
\put(117,34){$\displaystyle \prod _{D\in {\cal{D}}}\lambda _{D}$}
\end{picture}
\end{center}

$\displaystyle {\cal{L}}({\cal{P}}_{\cal{CD}})(f_{\cal{C}}(x))={\cal{L}}({\cal{P}}_{\cal{CD}})((\lambda _{C}(a_{\mu (C)}/C^{\top }))_{C\in {\cal{C}}})={\cal{L}}({\cal{P}}_{\cal{CD}})((\prod _{C\in {\cal{C}}})$\linebreak $\displaystyle ((a_{\mu (C)}/C^{\top })_{C\in {\cal{C}}}))=(\prod _{D\in {\cal{D}}}\lambda _{D})({\cal{P}}_{\cal{CD}}((a_{\mu (C)}/C^{\top })_{C\in {\cal{C}}}))=(\lambda _{D}(a_{\mu (D)}/$\linebreak $D^{\top }))_{D\in {\cal{D}}}$, where again $a_{\mu (D)}=a_{\mu (C)}$ iff $D\subseteq C$.

So $f_{\cal{D}}({\cal{Q}}_{\nu ({\cal{C}})\nu ({\cal{D}})}(x))= {\cal{L}}({\cal{P}}_{\cal{CD}})(f_{\cal{C}}(x))$ for all $x\in {\cal{L}}(A)_{\nu ({\cal{C}})}$ arbitrary, hence $f_{\cal{D}}\circ {\cal{Q}}_{\nu ({\cal{C}})\nu ({\cal{D}})}={\cal{L}}({\cal{P}}_{\cal{CD}})\circ f_{\cal{C}}$, that is we have the commutative diagram below.

\begin{center}
\begin{picture}(170,73)(0,0)
\put(0,30){${\cal{Q}}_{\nu ({\cal{C}})\nu ({\cal{D}})}$}
\put(45,45){\vector(0,-1){24}}
\put(40,10){${\cal{L}}(A)_{\nu ({\cal{D}})}$}
\put(40,48){${\cal{L}}(A)_{\nu ({\cal{C}})}$}
\put(84,15){\vector(1,0){24}}
\put(90,5){$f_{\cal{D}}$}
\put(84,53){\vector(1,0){24}}
\put(90,57){$f_{\cal{C}}$}
\put(110,10){${\cal{L}}(A_{\cal{D}})$}
\put(110,48){${\cal{L}}(A_{\cal{C}})$}
\put(115,45){\vector(0,-1){24}}
\put(117,30){${\cal{L}}({\cal{P}}_{\cal{CD}})$}
\end{picture}
\end{center}

Let us notice that we are situated in the conditions of Lemma \ref{izomlimind}. By the above, we have the inductive systems $(({\cal{L}}(A)_{\nu ({\cal{C}})})_{{\cal{C}}\in \Pi(A)},$\linebreak $({\cal{Q}}_{\nu ({\cal{C}})\nu ({\cal{D}})})_{{\cal{C}},{\cal{D}}\in \Pi(A),{\cal{C}}\leq {\cal{D}}})$ and $(({\cal{L}}(A_{\cal{C}}))_{{\cal{C}}\in \Pi(A)},({\cal{L}}({\cal{P}}_{\cal{CD}}))_{{\cal{C}},{\cal{D}}\in \Pi(A),{\cal{C}}\leq {\cal{D}}})$ in the category of bounded distributive lattices, and their inductive limits are $\widetilde{{\cal{L}}(A)}$ and ${\cal{L}}(\tilde{A})$, respectively. For each ${\cal{C}}\in \Pi(A)$, we have the bounded lattice isomorphism $f_{\cal{C}}:{\cal{L}}(A)_{\nu ({\cal{C}})}\rightarrow {\cal{L}}(A_{\cal{C}})$, and these isomorphisms verify: for all ${\cal{C}},{\cal{D}}\in \Pi (A)$ with ${\cal{C}}\leq {\cal{D}}$, $f_{\cal{D}}\circ {\cal{Q}}_{\nu ({\cal{C}})\nu ({\cal{D}})}={\cal{L}}({\cal{P}}_{\cal{CD}})\circ f_{\cal{C}}$.

Therefore, by Lemma \ref{izomlimind}, it follows that $\widetilde{{\cal{L}}(A)}$ and ${\cal{L}}(\tilde{A})$ are isomorphic bounded lattices (one isomorphism between them being $\varphi :\widetilde{{\cal{L}}(A)}\rightarrow {\cal{L}}(\tilde{A})$, for all ${\cal{C}}\in \Pi (A)$, for all $x\in {\cal{L}}(A)_{\nu ({\cal{C}})}$, $\varphi ([x])=[f_{\cal{C}}(x)]$, as Lemma \ref{izomlimind} shows).\end{proof}

\begin{example}
In this example we will determine the strongly co-Stone hull of the residuated lattice $A$ in Example \ref{lrex0,5}.\index{strongly co-Stone hull}

Let $B={\rm CoAnn}(A)$. $0^{\top }=a^{\top }=\{1\}$, $1^{\top }=A$, $b^{\top }=\{c,1\}=<c>$ and $c^{\top }=\{b,1\}=<b>$, hence $B=\{1^{\top },b^{\top },c^{\top },0^{\top }\}$ and $\Pi (A)=P(B)=\{\{1^{\top }\},\{b^{\top },c^{\top }\}\}$. Let ${\cal{C}}=\{1^{\top }\}$ and ${\cal{D}}=\{b^{\top },c^{\top }\}$. ${\cal{C}}\leq {\cal{D}}$. $A_{\cal{C}}=A/1^{\top }=A/A=\{1/A\}$ and $A_{\cal{D}}=A/b^{\top }\times A/c^{\top }$.

As the table of the operation $\leftrightarrow $ shows, $0/b^{\top }=\{0\}$, $a/b^{\top }=\{a,b\}=b/b^{\top }$ and $c/b^{\top }=1/b^{\top }=b^{\top }$, so $A/b^{\top }=\{0/b^{\top },a/b^{\top },1/b^{\top }\}$, and $0/c^{\top }=\{0\}$, $a/c^{\top }=\{a,c\}=c/c^{\top }$ and $b/c^{\top }=1/c^{\top }=c^{\top }$, so $A/c^{\top }=\{0/c^{\top },a/c^{\top },$\linebreak $1/c^{\top }\}$. Therefore $A_{\cal{D}}=\{0/b^{\top },a/b^{\top },1/b^{\top }\}\times \{0/c^{\top },a/c^{\top },1/c^{\top }\}=\{0,x_{0a},$\linebreak $x_{01},x_{a0},x_{aa},x_{a1},x_{10},x_{1a},1\}$, where we denoted: $0=(0/b^{\top },0/c^{\top })$, $1=(1/b^{\top },1/c^{\top })$ and $x_{ij}=(i/b^{\top },j/c^{\top })$ for all $i,j\in \{0,a,1\}$ with $(i,j)\notin \{(0,0),(1,1)\}$.

The strongly co-Stone hull of $A$ is $\displaystyle \tilde{A}=\varinjlim _{{\cal{E}}\in \Pi (A)}A_{\cal{E}}=A_{\cal{D}}$, because, as is easily seen, $(A_{\cal{D}},\{{\cal{P}}_{\cal{CD}},{\rm id}_{\cal{D}}\})$ is the inductive limit of the inductive system $(\{A_{\cal{C}},A_{\cal{D}}\},\{{\cal{P}}_{\cal{CD}}\})$.

The operations of $A_{\cal{D}}=\tilde{A}$ are defined componentwise from those of the quotient lattices $A/b^{\top }$ and $A/c^{\top }$, hence, like in $A$, $\odot =\wedge $ also in $\tilde{A}$. As shown by Lemma \ref{prop-cong}, (\ref{aF-leq-bF}), $A/b^{\top }$ and $A/c^{\top }$ share the same lattice structure, namely that of the three-element chain, hence the lattice structure of $A_{\cal{D}}=\tilde{A}$ is the following:

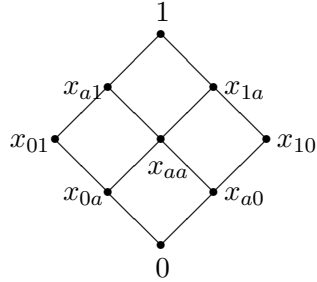
\begin{center}
\begin{picture}(120,110)(0,0)
\put(60,15){\circle*{3}}

\put(58,3){$0$}
\put(60,15){\line(-1,1){40}}
\put(40,35){\circle*{3}}
\put(23,32){$x_{0a}$}

\put(60,15){\line(1,1){40}}

\put(80,35){\circle*{3}}
\put(84,32){$x_{a0}$}
\put(40,35){\line(1,1){40}}
\put(80,35){\line(-1,1){40}}
\put(60,55){\circle*{3}}
\put(55,42){$x_{aa}$}
\put(20,55){\circle*{3}}
\put(3,52){$x_{01}$}
\put(40,75){\circle*{3}}
\put(23,72){$x_{a1}$}
\put(100,55){\circle*{3}}
\put(104,52){$x_{10}$}
\put(80,75){\circle*{3}}
\put(84,72){$x_{1a}$}
\put(60,95){\line(-1,-1){40}}

\put(60,95){\line(1,-1){40}}
\put(60,95){\circle*{3}}
\put(58,100){$1$}
\end{picture}
\end{center}

Here is the table of the operation $\rightarrow $ in $A_{\cal{D}}=\tilde{A}$:

\begin{center}
\begin{tabular}{c|ccccccccc}
$\rightarrow $ & $0$ & $x_{0a}$ & $x_{01}$ & $x_{a0}$ & $x_{aa}$ & $x_{a1}$ & $x_{10}$ & $x_{1a}$ & $1$\\ \hline
$0$ & $1$ & $1$ & $1$ & $1$ & $1$ & $1$ & $1$ & $1$ & $1$ \\
$x_{0a}$ & $0$ & $1$ & $1$ & $x_{10}$ & $1$ & $1$ & $x_{10}$ & $1$ & $1$ \\
$x_{01}$ & $0$ & $x_{1a}$ & $1$ & $x_{10}$ & $x_{1a}$ & $1$ & $x_{10}$ & $x_{1a}$ & $1$ \\
$x_{a0}$ & $0$ & $x_{01}$ & $x_{01}$ & $1$ & $1$ & $1$ & $1$ & $1$ & $1$ \\
$x_{aa}$ & $0$ & $x_{01}$ & $x_{01}$ & $x_{10}$ & $1$ & $1$ & $x_{10}$ & $1$ & $1$ \\
$x_{a1}$ & $0$ & $x_{0a}$ & $x_{01}$ & $x_{10}$ & $x_{1a}$ & $1$ & $x_{10}$ & $x_{1a}$ & $1$ \\
$x_{10}$ & $0$ & $x_{01}$ & $x_{01}$ & $x_{a1}$ & $x_{a1}$ & $x_{a1}$ & $1$ & $1$ & $1$ \\
$x_{1a}$ & $0$ & $x_{01}$ & $x_{01}$ & $x_{a0}$ & $x_{a1}$ & $x_{a1}$ & $x_{10}$ & $1$ & $1$ \\
$1$ & $0$ & $x_{0a}$ & $x_{01}$ & $x_{a0}$ & $x_{aa}$ & $x_{a1}$ & $x_{10}$ & $x_{1a}$ & $1$
\end{tabular}
\end{center}
\end{example}


\begin{thebibliography}{99}
\bibitem{bal} R. Balbes, P. Dwinger, {\em Distributive Lattices}, University of Missouri Press, Columbia, Missouri (1974).

\bibitem{bel1} L. P. Belluce, Semisimple Algebras of Infinite Valued Logic and Bold Fuzzy Set Theory, {\em Can. J. Math.} 38, No. 6 (1986), 1356-1379.

\bibitem{bel2} L. P. Belluce, Spectral Spaces and Non-commutative Rings, {\em Comm. Algebra} 19 (1991), 1855-1865.

\bibitem{bus} D. Bu\c{s}neag, {\em Categories of Algebraic Logic}, Editura Academiei Rom\^{a}ne, Bucure\c{s}ti (2006).

\bibitem{BusPic06} D. Bu\c{s}neag, D. Piciu, Residuated Lattices of Fractions Relative to a $\wedge $-closed System, {\em Bull. Math. Soc. Sci. Math. Roumanie (N.S.)}  49 (97) (2006), 13-24.

\bibitem{rcig} R. Cignoli, Free Algebras in Varieties of Stonean Residuated Lattices, {\em Soft Computing} 12 (2008), 315-320.

\bibitem{cre} R. Cre\c tan, A. Jeflea, On the Lattice of Congruence Filters of a Residuated Lattice, {\em Ann. Univ. Craiova, Math. Comp. Sci. Series} 33 (2006), 174-188.

\bibitem{dav} B. A. Davey, $m$-Stone Lattices, {\em Can. J. Math.} 24, No. 6 (1972), 1027-1032.

\bibitem{din3} A. Di Nola, G. Georgescu, Projective Limits of MV-spaces, {\em Order} 13 (1996), 391-398.

\bibitem{gal} N. Galatos, P. Jipsen, T. Kowalski, H. Ono, {\em Residuated Lattices: An Algebraic Glimpse at Substructural Logics}, Studies in Logic and The Foundations of Mathematics 151, Elsevier, Amsterdam/Boston/Heidelberg/London/New York/Oxford/Paris/San Diego/San Francisco/Singapore/Sydney/Tokyo (2007).

\bibitem{geo1} G. Georgescu, The Reticulation of a Quantale, {\em Rev. Roum. Math. Pures Appl.} 40, No. 7-8 (1995), 619-631.

\bibitem{eu3} G. Georgescu., L. Leu\c stean, C. Mure\c san, Maximal Residuated Lattices with Lifting Boolean Center, to appear in {\em Algebra Universalis}.

\bibitem{haj} P. ${\rm H\acute{a}jek}$, {\em Metamathematics of Fuzzy Logic}, Trends in Logic-Studia Logica, Kluwer Academic Publishers, Dordrecht/Boston/London (1998).

\bibitem{ior1} A. Iorgulescu, Classes of BCK Algebras-Part III, Preprint Series of the Institute of Mathematics of the Romanian Academy, preprint No. 3 (2004), 1-37.

\bibitem{ior2} A. Iorgulescu, Classes of BCK Algebras-Part IV, Preprint Series of the Institute of Mathematics of the Romanian Academy, preprint No. 4 (2004), 1-25.

\bibitem{ior} A. Iorgulescu, {\em Algebras of Logic as BCK Algebras}, Editura ASE, Bucharest (2008). 

\bibitem{kow} T. Kowalski, H. Ono, {\em Residuated Lattices: An Algebraic Glimpse at Logics without Contraction}, manuscript (2000). 

\bibitem{kuh} J. ${\rm K\ddot{u}hr}$, Boolean and Central Elements and Cantor-Bernstein Theorem in Bounded Pseudo-BCK-algebras, to appear in {\em  Journal of Multiple-valued Logic and Soft Computing}.


\bibitem{leo1} L. Leu\c{s}tean, The Prime and Maximal Spectra and the Reticulation of BL-algebras, {\em Central European Journal of Mathematics} 1, No. 3 (2003), 382-397.

\bibitem{leo} L. Leu\c{s}tean, {\em Representations of Many-valued Algebras}, Ph. D. Thesis, University of Bucharest (2004).

\bibitem{eu1} C. Mure\c{s}an, The Reticulation of a Residuated Lattice, {\em Bull. Math. Soc. Sci. Math. Roumanie} 51 (99), No. 1 (2008), 47-65.

\bibitem{eu2} C. Mure\c{s}an, Characterization of the Reticulation of a Residuated Lattice, to appear in {\em Journal of Multiple-valued Logic and Soft Computing}.

\bibitem{eu5} C. Mure\c{s}an, Further Functorial Properties of the Reticulation, {\em Journal of Multiple-valued Logic and Soft Computing} 16, No. 1-2 (2010), 177-187.

\bibitem{pic} D. Piciu, {\em Algebras of Fuzzy Logic}, Editura Universitaria Craiova, Craiova (2007).

\bibitem{sim} H. Simmons, Reticulated Rings, {\em J. Algebra} 66 (1980), 169-192.

\bibitem{tur} E. Turunen, {\em Mathematics behind Fuzzy Logic}, Advances in Soft Computing, Physica-Verlag, Heidelberg (1999).
\end{thebibliography}
\end{document}